\documentclass[ejs]{imsart}

\pubyear{2006}
\volume{0}
\issue{0}
\firstpage{1}
\lastpage{8}

\RequirePackage[OT1]{fontenc}
\RequirePackage{amsthm,amsmath}
\RequirePackage[numbers]{natbib}
\RequirePackage[colorlinks,citecolor=blue,urlcolor=blue]{hyperref}

\usepackage{accents}
\usepackage{amssymb}
\usepackage{bm}
\usepackage{booktabs}
\usepackage{caption}
\usepackage{cases}
\usepackage{color}
\usepackage{enumitem}
\usepackage{float}
\usepackage{graphicx}
\usepackage{graphics}
\usepackage{latexsym}
\usepackage{lscape}
\usepackage{mathrsfs}
\usepackage{multirow}
\usepackage{moresize}
\usepackage{pstricks}
\usepackage{pst-grad} 
\usepackage{pst-plot} 
\usepackage{syntonly}
\usepackage{subfig}
\usepackage[toc,page]{appendix}
\usepackage{verbatim}
\usepackage{rotating}

\newcommand{\ubar}[1]{\underaccent{\bar}{#1}}

\newcommand{\vA}[0]{\mathbf A}

\newcommand{\ve}[0]{\mathbf e}

\newcommand{\vu}[0]{\mathbf u}
\newcommand{\vv}[0]{\mathbf v}
\newcommand{\vw}[0]{\mathbf w}
\newcommand{\vx}[0]{\mathbf x}
\newcommand{\vX}[0]{\mathbf X}

\newcommand{\vz}[0]{\mathbf z}

\newcommand{\vzero}[0]{\mathbf 0}
\newcommand{\vone}[0]{\mathbf 1}
\newcommand{\mbf}[1]{\mbox{\boldmath$#1$}}
\newcommand{\Cov}[0]{\text{Cov}}

\newcommand{\diag}[0]{\text{diag}}
\newcommand{\tr}[0]{\text{tr}}

\newcommand{\calI}[0]{\mathcal{I}}

\newcommand{\E}[0]{\mathbb{E}}

\newcommand{\supp}[0]{\text{supp}}

\newcommand{\Prob}[0]{\mathbb{P}}
\newcommand{\I}[0]{\mathbb{I}}

\newcommand{\FPR}[0]{\text{FPR}}
\newcommand{\FNR}[0]{\text{FNR}}

\startlocaldefs
\numberwithin{equation}{section}
\theoremstyle{plain}
\newtheorem{thm}{Theorem}[section]
\endlocaldefs

\newtheorem{lem}[thm]{Lemma}

\theoremstyle{definition}
\newtheorem{defn}{Definition}[section]

\newtheorem{exmp}{Example}[section]

\theoremstyle{remark}

\begin{document}

\begin{frontmatter}
\title{Sparse transition matrix estimation for high-dimensional and locally stationary vector autoregressive models\thanksref{t1}}
\runtitle{High-dimensional VAR}

\begin{aug}
\author{\fnms{Xin} \snm{Ding$^\ddagger$,}\ead[label=e1]{xin.ding@stat.ubc.ca}}
\author{\fnms{Ziyi} \snm{Qiu$^\dagger$$^\star$,}\ead[label=e2]{ziyiqiu@illinois.edu}}
\author{\fnms{Xiaohui} \snm{Chen$^\dagger$}\ead[label=e3]{xhchen@illinois.edu }}

\address{$^\ddagger$University of British Columbia\\
$^\dagger$University of Illinois at Urbana-Champaign\\
$^\star$University of Chicago\\
\printead{e1,e2,e3}}

\thankstext{t1}{Research partially supported by NSF grant DMS-1404891 and UIUC Research Board Award RB15004.}
\runauthor{X. Ding et al.}

\affiliation{University of Illinois at Urbana-Champaign}
\end{aug}

\begin{abstract}
We consider the estimation of the transition matrix in the high-dimensional time-varying vector autoregression (TV-VAR) models. Our model builds on a general class of locally stationary VAR processes that evolve smoothly in time. We propose a hybridized kernel smoothing and $\ell^1$-regularized method to directly estimate the sequence of time-varying transition matrices. Under the sparsity assumption on the transition matrix, we establish the rate of convergence of the proposed estimator and show that the convergence rate depends on the smoothness of the locally stationary VAR processes only through the smoothness of the transition matrix function. In addition, for our estimator followed by thresholding, we prove that the false positive rate (type I error) and false negative rate (type II error) in the pattern recovery can asymptotically vanish in the presence of weak signals without assuming the minimum nonzero signal strength condition. Favorable finite sample performances over the $\ell^2$-penalized least-squares estimator and the unstructured maximum likelihood estimator are shown on simulated data. We also provide two real examples on estimating the dependence structures on financial stock prices and economic exchange rates datasets.
\end{abstract}


\tableofcontents
\end{frontmatter}

\section{Introduction}
\label{sec:introduction}

Vector autoregression (VAR) is a basic tool in multivariate time series analysis and it has been extensively used to model the cross-sectional and serial dependence in various applications from economics and finance \cite{sim1980,Backus1986,baillielippensmcmahon1983,campbellshiller1988,sims1992,cochrane1994,stockwatson2005,Coad2010,anwang2012}. There are two fundamental limitations of the vector-autoregressive models. First, conventional methods to estimate the transition matrix of a VAR model are based on the least squares (LS) estimator and the maximum likelihood estimator (MLE), in which the parameter estimation is consistent when the sample size increases and the model size is {\it fixed} \cite{brockwelldavis1991,lutkepohl2007}. Since the number of parameters grows quadratically in the number of time series variables, the VAR model typically includes no more than ten variables in many real applications \cite{fanhanliu2014}. However, due to the recent explosive data enrichment, analysis of panel data with a few hundred variables is often encountered, in which the LS and MLE are not suitable even for a moderate problem size. Second, stationarity plays a major role in the VAR model. Therefore, the stationary VAR does not capture the time-varying underlying data generation structures, which have been observed in a broad regime of applications in economics and finance \cite{angpiazzesi2003,Primiceri2005}.

Motivated by the limitations of the VAR, this paper studies the estimation problems of the time-varying VAR (TV-VAR) model for high-dimensional time series data. Let $\vX_{d \times n} = (\vx_1,\cdots,\vx_n)$ be a sequence of $d$-dimensional observations generated by a mean-zero TV-VAR of order 1 (TV-VAR(1))
\begin{equation}
\label{eqn:tvvar}
\vx_i =  A(i/n) \vx_{i-1}  + \ve_i,
\end{equation}
where $A(t), t \in [0,1]$, is a $d \times d$ matrix-valued deterministic function consisting of the transition matrices $A_i:=A(i/n)$ at evenly spaced time points and $\ve_i$ are independent and identically distributed (iid) mean-zero random errors, i.e. innovations. In this paper, our main focus is to estimate the transition matrices $A_i$ for the TV-VAR(1) and the extension to higher-order VAR is straightforward. Indeed, for a general TV-VAR of order $k \ge 1$
\begin{equation*}
\vx_i =  A_{i,1} \vx_{i-1} + A_{i,2} \vx_{i-2} + \cdots + A_{i,k} \vx_{i-k}  + \ve_i,
\end{equation*}
we can rewrite $\vz_i = (\vx_i^\top, \cdots, \vx_{i-k+1}^\top)^\top$ at time $i = k, k+1,\cdots, n$, as a TV-VAR(1) in the augmented space
\begin{equation*}
\vz_i = \vA_i \vz_{i-1} + \tilde\ve_i,
\end{equation*}
where
\begin{equation*}
\vA_i = \left(
\begin{array}{ccccc}
A_{i,1} & A_{i,2} & \cdots & A_{i,k-1} & A_{i,k} \\
I_{d \times d} & \vzero_{d \times d} & \cdots & \vzero_{d \times d} & \vzero_{d \times d} \\
\vzero_{d \times d} & I_{d \times d} & \cdots & \vzero_{d \times d} & \vzero_{d \times d} \\
\vdots & \vdots & \vdots & \vdots & \vdots \\
\vzero_{d \times d} & \vzero_{d \times d} & \cdots & I_{d \times d} & \vzero_{d \times d} \\
\end{array}
\right) , \qquad \tilde\ve_i = \left(
\begin{array}{c}
\ve_i \\
\vzero_{d \times 1} \\
\vdots \\
\vzero_{d \times 1}
\end{array} \right) ,
\end{equation*}
$I_{d \times d}$ is the $d \times d$ identity matrix and $\vzero_{d \times d'}$ is the $d \times d'$ zero matrix. Then, we need only to estimate the first $d$ rows in $\vA_i$.

To make the estimation problem feasible for high-dimensional time series when $d$ is large, it is crucial to carefully regularize the coefficient matrices $A_i$. The main idea is to use certain low-dimensional structures in $A_i$ and the degree of freedom of (\ref{eqn:tvvar}) can be dramatically reduced. In our problem, we assume two key structures in $A_i$. First, since our goal is to estimate a sequence of transition matrices, which can be viewed as the discrete version of a matrix-valued function, it is natural to impose the {\it smoothness} condition to $A(t)$. In this case, (\ref{eqn:tvvar}) is closely related to the {\it locally stationary processes}, a general class of non-stationary processes proposed by \cite{dahlhaus1997a}. Examples of other linear and nonlinear locally stationary processes can be found in \cite{chenxuwu2013}. In particular, let $i_0 = 1,\cdots,n$, be any time point and $t_0=i_0/n$. Then, under suitable regularity conditions \cite{Dahlhaus2012}, there exists a stationary process $\tilde{\vx}_i(t_0)$ such that for all $j=1,\cdots,d$,
$$
|X_{ij} - \tilde{X}_{ij}(t_0)| = O_\Prob(|i/n-t_0| + 1/n),
$$
where $X_{ij}$ and $\tilde{X}_{ij}(t_0)$ are the $j$-th element of ${\vx}_i$ and $\tilde{\vx}_i(t_0)$ respectively, and
$$
\tilde{\vx}_i(t_0) = A(t_0) \tilde{\vx}_{i-1}(t_0) + \ve_i.
$$
Therefore, $\vx_i$ is an approximately stationary VAR process $\tilde{\vx}_i(t_0)$ in a small neighborhood of $t_0$.

Second, at each time point $i=1,\cdots,n$, we need to estimate a $d \times d$ matrix. There have been different structural options in literature, such as the sparse \cite{hanluliu2015,wanglitsai2007a,songbickel2011,shojaiemichailidis2010}, banded \cite{guowangyao2015}, and low-rank \cite{hanxuliu2016} transition matrix, all of which only considered the stationary VAR model. In this paper, we consider a sequence of {\it sparse} transition matrices and we allow the sparsity patterns to change over time. Note that our problem is also different from the emerging literature for estimating the high-dimensional covariance matrix and its related functionals of time series data \cite{chenxuwu2013,chenxuwu2016,zhoulaffertywasserman2010a,basumichailidis2015}, since our goal is to directly estimate the data generation mechanism specified by $A(\cdot)$.

To simultaneously address the two issues, we propose a hybridized method of the nonparametric smoothing technique and $\ell^1$-regularization to estimate the sparse transition matrix of the locally stationary VAR. The proposed method is equivalent to solving a sequence of large numbers of linear programs and therefore the estimates can be efficiently obtained by using the high-performance (i.e. parallel) computing technology; see Section \ref{sec:method} for details. In Section \ref{sec:asymptotics}, we establish the rate of convergence under suitable assumptions on the smoothness of the covariance function and the sparsity of the transition matrix function $A(\cdot)$. Specifically, the dimension $d$ is permitted to increase sub-exponentially fast in the sample size $n$ to obtain consistent estimation, i.e. $d = o(\exp(n))$. In addition, we also prove that when our estimator is followed by thresholding, type I and type II errors in the pattern recovery asymptotically vanish in the presence of weak signals. In contrast with the existing literature on consistent model selection, we do not require the minimum nonzero signal strength to be bounded away from zero. Simulation studies in Section \ref{sec:simulation} and two real data examples in Section \ref{sec:real-data-analysis} demonstrate favorable performances and more interpretable results of the proposed TV-VAR model. Technical proofs are deferred to Section \ref{sec:proofs}.

We fix the notations that will be used in the rest of the paper. Let $M$ be a generic matrix, $\calI \subset \mathbb{N}$ is an index subset, $[M]_{*\calI}$ and $[M]_{\calI*}$ be the submatrix of $M$ with columns and rows indexed by $\cal I$, resp. Denote $J = \{1,\cdots,d\}$. Let $\rho(M) = \sup_{|\vv|=1} |M \vv|$ be the spectral norm of $M$. $|M|_1, |M|_\infty$ and $|M|_F$ are the entry $\ell^1$, $\ell^\infty$ and Frobenius norm of $M$, resp. $|M|_{\ell^1} = \max_{j \le d} \sum_{k=1}^d |M_{jk}|$ and $|M|_{\ell^\infty} = \max_{k \le d} \sum_{j=1}^d |M_{jk}|$ are the matrix $\ell^1$ and $\ell^\infty$ norm of $M$, resp. $\supp(M) = \{(\ell, j) : M_{\ell j} \neq 0\}$ is the support of $M$ and $|S|$ is the number of nonzeros in the support set $S$. For $q > 0$ and a random variable (r.v.) $X$, $\|X\|_q = (\E|X|^q)^{1/q}$ and we say that $X \in {\cal L}^q$ iff $\|X\|_q < \infty$. Write $\|X\|=\|X\|_2$. Denote $a \wedge b = \min(a, b)$ and $a \vee b = \max(a, b)$. If $a$ and $b$ are vectors, then the maximum and minimum are interpreted element-wise.

\section{Method and algorithm}
\label{sec:method}

For $m \in \mathbb{Z}$, let $\Sigma_{i,m} = \Cov(\vx_i, \vx_{i+m}) = \E(\vx_i \vx_{i+m}^\top)$. Then
\begin{equation}
\label{eqn:key-observation}
\Sigma_{i-1,1} = \Sigma_{i-1,0} A_i^\top \quad  \text{and} \quad \Sigma_{i,-1} = A_i \Sigma_{i-1,0}.
\end{equation}
Therefore, for any estimator, say $\hat{A}_i$, that is reasonably close to the true coefficient matrices $A_i$, we must have $(\Sigma_{i-1,1} -\Sigma_{i-1,0} \hat{A}_i^\top ) \approx \vzero_{d \times d}$ and $(\Sigma_{i,-1} - \hat{A}_i \Sigma_{i-1,0}) \approx \vzero_{d \times d}$. A naive estimator for $A_i$ would be constructed to invert the sample versions of (\ref{eqn:key-observation}). Estimators of this kind do not have good statistical properties in high-dimensions because of the ill-conditioning and the dependence information in $A_i$ is not directly used in estimation. If $A_i$ is known to be sparse as {\it a priori}, we may consider the following constrained minimization program
\begin{eqnarray}
\label{eqn:tvvar-clime_oracle}
\text{minimize}_{ \Lambda \in \mathbb{R}^{d \times d} } && |\Lambda|_{1} \\ \nonumber
\text{subject to} && | \Sigma_{i-1,1} -\Sigma_{i-1,0} \Lambda |_\infty \le \tau \\ \nonumber
&& | \Sigma_{i,-1} - \Lambda^\top \Sigma_{i-1,0}|_\infty \le \tau.
\end{eqnarray}
Because $\Sigma_{i,1}$ and $\Sigma_{i,0}$ are unknown, the solution of (\ref{eqn:tvvar-clime_oracle}) is an {\it oracle} estimator and therefore it is not implementable in practice.

Let $A(\cdot)$ be the continuous version of $A_i, i =1,\cdots,n$, in (\ref{eqn:tvvar}), and fix a $t \in (0,1)$. To estimate $A(t)$, we first estimate $ \Sigma_{i,1}$ and $\Sigma_{i,0}$ in (\ref{eqn:key-observation}) by their empirical versions. Let
\begin{equation}
\hat{\Sigma}_{t,\ell} = \sum_{m=1}^n w(t,m) \vx_m \vx_{m+\ell}^\top, \qquad \ell = -1, 0, 1,
\end{equation}
be the kernel smoothed estimator of $ \Sigma_{i,0}, \Sigma_{i,1}$ and $\Sigma_{i,-1}$, where $w(t,m)$ are the nonnegative weights. Here, we consider the Nadaraya-Watson smoothed estimator
\begin{equation}
w(t,m) = {K({t-m/n \over b_n}) \over \sum_{\ell=1}^n K({t - \ell/n \over b_n })},
\end{equation}
where $K(\cdot)$ is a nonnegative bounded kernel function with support in $[-1,1]$ such that $\int_{-1}^1 K(v) dv = 1$ and $b_n$ is the bandwidth parameter. We assume throughout the paper that the bandwidth satisfies the natural conditions: $b_n = o(1)$ and $n^{-1} = o(b_n)$. Then, our estimator $\hat{A}(t)$ is defined as the transpose of the solution of
\begin{eqnarray}
\label{eqn:tvvar-clime}
\text{minimize}_{\Lambda \in \mathbb{R}^{d \times d} } && |\Lambda|_{1} \\ \nonumber
\text{subject to} && | \hat{\Sigma}_{i-1,1} -\hat{\Sigma}_{i-1,0} \Lambda |_\infty \le \tau \\ \nonumber
&& | \hat{\Sigma}_{i,-1} - \Lambda^\top \hat{\Sigma}_{i-1,0}|_\infty \le \tau.
\end{eqnarray}
Observe that (\ref{eqn:tvvar-clime}) is equivalent to the following $d$ optimization sub-problems
\begin{eqnarray}
\label{eqn:tvvar-subproblems}
\hat{\mbf\beta}_j(t) &=& \arg\min_{ \vu \in \mathbb{R}^d} |\vu|_1 \\ \nonumber
\text{subject to} && |[\hat{\Sigma}_{i-1,1}]_{*j} - \hat{\Sigma}_{i-1,0} \vu |_\infty \le \tau \\ \nonumber
&& | [\hat{\Sigma}_{i,-1}]_{j*} - \vu^\top \hat{\Sigma}_{i-1,0}|_\infty \le \tau
\end{eqnarray}
in that $[\hat{A}(t)]_{j*} = \hat{\mbf\beta}_j(t)^\top$. Since the $d$ sub-problems (\ref{eqn:tvvar-subproblems}) can be independently solved, we can efficiently compute the solution of (\ref{eqn:tvvar-clime}) by parallelizing the optimizations of (\ref{eqn:tvvar-subproblems}). In addition, each sub-problem in (\ref{eqn:tvvar-subproblems}) can be recasted as a linear program (LP)
\begin{eqnarray*}
\text{minimize}_{\vv, \vu \in \mathbb{R}^d_+} &  & |\vv|_1 + |\vw|_1 \\
\text{subject to} & &  -\hat{\Sigma}_{i-1,0} \vv + \hat{\Sigma}_{i-1,0} \vw \le \tau - \max\{ [\hat{\Sigma}_{i-1,1}]_{*j}, [\hat{\Sigma}_{i-1,-1}]_{*j}^\top \} \\
& & \hat{\Sigma}_{i,0} \vv - \hat{\Sigma}_{i-1,0} \vw \le \tau + \min \{ [\hat{\Sigma}_{i,1}]_{*j},  [\hat{\Sigma}_{i,-1}]_{*j}^\top \}
\end{eqnarray*}
so that $\hat{\mbf\beta}_j(t) = \hat{\vv} - \hat{\vw}$, where $(\hat{\vv}, \hat{\vw})$ is the nonnegative solution of the LP. Note that (\ref{eqn:tvvar-subproblems}) can be efficiently solved by the simplex algorithm \cite{murty1983linear}. Compared with the estimation of sparse transition matrix for the stationary VAR model \cite{hanluliu2015}, our estimator requires to solve two sample versions of Yule-Walker equations in (\ref{eqn:key-observation}) and (\ref{eqn:tvvar-clime}) due to the non-stationarity and therefore two-sided constraints of the estimator $\hat{A}(t)$ are needed in (\ref{eqn:tvvar-clime}).

\section{Asymptotic properties}
\label{sec:asymptotics}

\subsection{Locally stationary VAR process in $L^2$}

To establish an asymptotic theory for estimating the continuous matrix-valued transition function $A(\cdot)$ of the non-stationary VAR, it is more convenient to model the time series as realizations from a continuous $d$-dimensional process. Following the framework of \cite{draghicescuguillaswu2009a}, $\vx_1,\cdots,\vx_n$ are viewed as realizations of the continuous $d$-dimensional process $\vx(t)=(X_1(t),\cdots,X_d(t))^\top, t \in [0,1],$ on the discrete grid $t_i=i/n,i=1,\cdots,n$, i.e. $\vx(i/n)=A(i/n)\vx((i-1)/n)+\ve(i/n)$ and the temporal dependence of $\vx_i$ is carried out on the rescaled time indices $t_i$.

\begin{defn}[Locally stationary VAR process in $L^2$]
\label{def:locally-stationary-L2-process}
A mean-zero non-stationary VAR process $\{\vx_i; i =1,\cdots,n\}$ in $\mathbb{R}^d$ of the form (\ref{eqn:tvvar}) is said to be {\it locally stationary in $L^2$}, or {\it weakly locally stationary}, if for every $i=1,\cdots,n$ and $t=i/n$, there exists a mean-zero stationary VAR process $\{\tilde{\vx}_m(t); m=1,\cdots,n\}$ given by
\begin{equation}
\label{eqn:approx_stat-var}
\tilde{\vx}_m(t) = A(t) \tilde{\vx}_{m-1}(t) + \ve_m
\end{equation}
such that for all $m=1,\cdots,n,$
\begin{equation}
\label{eqn:locally-stat-condition}
\max_{1 \le j \le d} \|\tilde{X}_{mj}(t) - X_{mj}\| \le C \left(\left|{m \over n}-t\right| + {1\over n}\right)
\end{equation}
and the constant $C>0$ does not depend on $d,m,$ and $n$.
\end{defn}

Note that the approximating stationary VAR process in Definition \ref{def:locally-stationary-L2-process} generally depends on $t$. As suggested by the Yule-Walker equations (\ref{eqn:key-observation}), given a fixed time of interest, estimation of $A(t)$ relies on an estimate of $\Sigma_{t,\ell}$. Consider $\ell=0$ and let $\Sigma(t) = \Sigma_{t,0}$ be the matrix-valued covariance function at lag-zero. With a finite number of observations $\vx_1,\cdots,\vx_n$, it is unclear how to define the covariance function $\Sigma(\cdot)$ off the $n$ points $t_i=i/n$. Nevertheless, the weakly locally stationary VAR processes provide a natural framework for extending $\Sigma(t_i)$ to $\Sigma(t)$ for all $t \in (0,1)$.

Since $A(\cdot)$ is continuous on $[0,1]$, the stationary VAR process in (\ref{eqn:approx_stat-var}) is defined for all $t \in (0,1)$. Let $\tilde\Sigma(t) = \E[\tilde{\vx}_m(t) \tilde{\vx}_m(t)^\top]$. By (\ref{eqn:locally-stat-condition}) and the Cauchy-Schwarz inequality, we have for each $j,k=1,\cdots,d,$
\begin{eqnarray*}
|\sigma_{jk}(t_i) - \tilde\sigma_{jk}(t_i)| &=& |\E(X_{ij}X_{ik})-\E(\tilde{X}_{ij}(t_i)\tilde{X}_{ik}(t_i))| \\
&\le& \|X_{ij}\|_2 \cdot \|X_{ik}-\tilde{X}_{ik}(t_i)\| + \|\tilde{X}_{ik}(t_i)\| \cdot \|X_{ij}-\tilde{X}_{ij}(t_i)\| \\
&\le& (\|X_{ij}\|_2 + \|\tilde{X}_{ik}(t_i)\|) \max_{1 \le j \le d} \|X_{ij}-\tilde{X}_{ij}(t_i)\| \\
&\le& C n^{-1},
\end{eqnarray*}
where the constant $C$ here is uniform in $j,k=1,\cdots,d$ and $i=1,\cdots,n$. Therefore, we get
$$
\Sigma(t_i) = \tilde\Sigma(t_i) + O(n^{-1}).
$$
Letting $n \to \infty$ and by the continuity of $A(\cdot)$, we can extend $\Sigma(t) = \tilde\Sigma(t)$ for all $t \in (0,1)$. Similar extension can be done for $\Sigma_{t,\ell}$ for $\ell=\pm1$. In Section \ref{subsec:rate-of-convergence}, it will be shown that the asymptotic theory of estimating $A(t), t \in (0,1)$ depends on the smoothness of the weakly locally stationary VAR processes only through the smoothness of $\Sigma(t)$ and therefore $A(t)$.

\subsection{Rate of convergence}
\label{subsec:rate-of-convergence}

In this section, we characterize the rate of convergence of our estimator (\ref{eqn:tvvar-clime}) under various matrix norms. We assume $d \ge 2$. To study the asymptotic properties of the proposed estimator, we make the following assumptions on model (\ref{eqn:tvvar}).
\begin{enumerate}
\item The coefficient matrices are sparse: let $0 \le \alpha < 1$, for each $i=1,\cdots,n,$
\begin{equation}
\label{eqn:sparse-matrix}
A_{i} \in {\cal G}_\alpha(s, M_d) = {\cal G}_\alpha^{|}(s, M_d) \cap {\cal G}_\alpha^{-}(s, M_d),
\end{equation}
where
\begin{eqnarray*}
{\cal G}_\alpha^{\vert}(s, M_d) &=& \Big \{ \Theta \in \mathbb{R}^{d \times d} : \max_{j \le d} \sum_{\ell=1}^d |\Theta_{\ell j}|^\alpha \le s, |\Theta|_{\ell^1} \le M_d \Big \}, \\
{\cal G}_\alpha^{-}(s, M_d) &=& \Big \{ \Theta \in \mathbb{R}^{d \times d} : \max_{\ell \le d} \sum_{j=1}^d |\Theta_{\ell j}|^\alpha \le s, |\Theta|_{\ell^\infty} \le M_d \Big \}.
\end{eqnarray*}
\item The marginal and lag-one covariance matrix processes $\{\Sigma_0(t)\}_{t \in [0,1]}$ and $\{\Sigma_1(t)\}_{t \in [0,1]}$ are smooth functions in $\mathbb{R}^{d \times d}$: $\Sigma_{\ell,jk}(t) \in {\cal C}^2([0,1]), \ell = 0,1$, where ${\cal C}^2([0,1])$ is the class of functions defined on $[0,1]$ that are twice differentiable with bounded derivatives uniformly in $j,k=1,\cdots,d$.
\item The random innovations $\ve_i = (e_{i1},\cdots,e_{id})^\top$ have iid components and sub-Gaussian tail: $\Prob(|e_{ij}| \ge x) \le C e^{-c x^2}$ for all $x > 0$.
\end{enumerate}

Before proceeding, we discuss the above assumptions. (\ref{eqn:sparse-matrix}) requires that the transition matrices are sparse in both columns and rows at all time points. A similar matrix class defined by (\ref{eqn:sparse-matrix}) was first proposed in \cite{bickellevina2008a} for {\it symmetric} matrices and it has been widely used for estimating high-dimensional covariance and precision matrix; see e.g. \cite{cailiuluo2011a,chenxuwu2013}. If $\alpha =0$, then the maximum number of nonzeros in columns and rows of $A_i$ is at most $s$. Assumption 2 requires that the marginal and lag-one covariance matrices evolve smoothly in time. The smoothness is not defined directly on $A(\cdot)$ for the ease of theorem statements. In view of (\ref{eqn:locally-stat-condition}), Assumption 2 is implied by the smoothness on $A_i$ under extra regularity conditions. For a generic matrix $M(t)$ parameterized by $t \in (0,1)$, we let $\dot{M}(t)$ and $\ddot{M}(t)$ be the first two element-wise derivatives of $M(t)$ w.r.t. $t$.

\begin{lem}
\label{lem:smoothness-relation}
Assume that $\sup_{t \in [0,1]} |A(t)|_{\ell^1} < 1$. Then we have
\begin{equation}
\label{eqn:smoothness-relation-1dev}
|\dot{\Sigma}(t)|_\infty \le {2 |\dot{A}(t)|_\infty \cdot |A(t)|_{\ell^1} \cdot |\Sigma(t)|_{\ell^1} \over 1 - |A(t)|_{\ell^1}^2}
\end{equation}
and
\begin{eqnarray}
\label{eqn:smoothness-relation-2dev}
|\ddot{\Sigma}(t)|_\infty &\le& {8 |\dot{A}(t)|_\infty \cdot |A(t)|_{\ell^1}^2 \cdot |\dot{A}(t)|_{\ell^1} \cdot |\Sigma(t)|_{\ell^1} \over (1 - |A(t)|_{\ell^1}^2)^2} \\ \nonumber
&& + {4 \max\{  |\dot{A}(t)|_\infty, \;  |\ddot{A}(t)|_\infty \} \cdot \max\{ |A(t)|_{\ell^1}, \; |\dot{A}(t)|_{\ell^1}\} \cdot  |\Sigma(t)|_{\ell^1}  \over 1 - |A(t)|_{\ell^1}^2}.
\end{eqnarray}
\end{lem}

By Lemma \ref{lem:smoothness-relation}, if $\sup_{t \in [0,1]}|A(t)|_{\ell^1} < 1$ and $A_{jk}(\cdot) \in {\cal C}^2([0,1])$ such that $\sup_{t \in [0,1]}|\dot{A}(t)|_{\ell^1} \vee |\Sigma(t)|_{\ell^1} \le K$ for some constant $K$, then $\Sigma_{jk}(\cdot)$ is also a ${\cal C}^2([0,1])$ function and therefore Assumption 2 is fulfilled.

Assumption 3 specifies the tail probability of the innovations $\ve_i$. In \cite{hanluliu2015}, $\ve_i$'s follow iid $N(\vzero, \Psi)$ for some error covariance matrix $\Psi$. A simple transformation by $\Psi^{-1/2}$ will reduce to the case that $\ve_i$ has iid components with the standard normal distribution, a special case of Assumption 3, which covers the sub-Gaussian innovations.

\begin{thm}
\label{thm:rate-subGaussian-instantaneous}
Let $\bar{\rho} = \sup_{i \ge 0} \rho(A_i)$. Fix an $i \in [nb_n+1,n(1-b_n)+1]$. Suppose that Assumption 1,2,3 are satisfied, $\bar{\rho} < 1$ and $A_i \in {\cal G}_\alpha(s, M_d)$. Then, with probability at least $1 - 2 d^{-1}$, we have the estimator $\hat{A}_i$ with tuning parameter $\tau \ge C (1 + M_d) (b_n^2 + \sqrt{(\log d)/ (n b_n)})$ obeys
\begin{eqnarray}
\label{eqn:rate-instantaneous-max}
|\hat{A}_{i} - A_{i}|_\infty &\le& 2 \tau |\Sigma_{i-1,0}^{-1}|_{\ell^1}, \\
\label{eqn:rate-instantaneous-spectral}
\rho(\hat{A}_{i} - A_{i}) &\le& C(\alpha) s (\tau |\Sigma_{i-1,0}^{-1}|_{\ell^1})^{1-\alpha}, \\
\label{eqn:rate-instantaneous-F}
d^{-1} |\hat{A}_{i} - A_{i}|_F^2 &\le& C(\alpha) s (\tau |\Sigma_{i-1,0}^{-1}|_{\ell^1})^{2-\alpha}.
\end{eqnarray}
\end{thm}

From Theorem \ref{thm:rate-subGaussian-instantaneous}, the bandwidth $b_n \asymp ((\log{d})/n)^{-1/5}$ gives the optimal rate of convergence and the resulting tuning parameter $\tau \ge C (1+M_d) ((\log{d})/n)^{-2/5}$. When $d$ is fixed, it is known that the optimal bandwidth in the nonparametric kernel density estimation is $n^{-1/5}$ for twice continuously differentiable functions under the mean integrated square error. So in the high-dimensional context, the dimension only has a logarithm impact on the choice of optimal bandwidth.

\subsection{Pattern recovery}

We also study the recovery of time-varying patterns using the estimator (\ref{eqn:tvvar-clime}). Let $S_{i} = \supp(A_{i})$ be the nonzero positions of $A_{i}$. If the nonzero entries of $A_{i}$ are small enough, then it is impossible to accurately distinguish the small nonzeros and the zeros. Therefore, the best hope we can expect is that the nonzero entries of $A_{i}$ with large magnitudes can be well separated from the zeros in $A_i$. Let
\begin{equation}
\label{eqn:u_sharp}
u_\sharp = 2 \tau |\Sigma_{i-1,0}^{-1}|_{\ell^1},
\end{equation}
where $\tau$ is determined in Theorem \ref{thm:rate-subGaussian-instantaneous}. We use the thresholded version of (\ref{eqn:tvvar-clime}) as an estimator of $S_{i}$
\begin{equation}
\label{eqn:tvvar-clime-thresholded}
\hat{S}_{i} = T_{u_\sharp}(\hat{A}_{i}) = \{\vone(|\hat{A}_{i,jk}| > u_\sharp)\}_{j,k=1}^d.
\end{equation}
By Theorem \ref{thm:rate-subGaussian-instantaneous}, the maximal fluctuation $|\hat{A}_{i} - A_{i}|_\infty$ is controlled by $u_\sharp$ with probability at least $1 - 2 d^{-1}$. So if we apply a thresholding procedure for $\hat{A}_i$ at the level $u_\sharp$, then we expect that the zeros and non-zeros with magnitudes larger than $2 u_\sharp$ in $A_i$ can be identified by the thresholded support of $\hat{A}_i$. Precisely, we have the instantaneous {\it partial} recovery consistency.

\begin{thm}
\label{thm:partial-pattern-recovery}
Assume that Assumption 1,2,3 are satisfied. Then, for any fixed time points $i \in [nb_n+1, n(1-b_n)+1]$, $\Prob(\hat{S}_{i} \subset S_{i}) \ge 1 - 2 d^{-1}$ and $\Prob(\{ (m, k) : |A_{i,mk}| > 2 u_\sharp \} \subset \hat{S}_{i}) \ge 1 - 2 d^{-1}$.
\end{thm}

Theorem \ref{thm:partial-pattern-recovery} states that, with high probability, the zeros in $A_i$ can be identified and the nonzero entries in $A_i$ with strong signal strength above $2 u_\sharp$ can be recovered by $\hat{S}_i$. Therefore, the false positives (type I error) of the estimator (\ref{eqn:tvvar-clime-thresholded}) are asymptotically controlled; see Theorem \ref{eqn:pattern-recovery-consistency} for precise statement. However, Theorem \ref{thm:partial-pattern-recovery} does not provide too much information regarding the false negatives since there is no characterization of the signal strength in $(0,2 u_\sharp)$.

Let $\beta > 0$ and $u_0 \in (0,1)$. We introduce the following $d \times d$ matrix class

\begin{align}
&{\cal G}_{\alpha, \beta}(s, M_d, L_d)   \nonumber\\
=&{\cal G}_{\alpha}(s, M_d) \cap \left\{ \Theta: { |\{(m, k) : |\Theta_{mk}| \in (0, u) \} | \over |\supp(\Theta)| } \le L_d u^\beta, \; \forall 0 < u < u_0 \right\}.
\end{align}

For $A \in {\cal G}_{\alpha, \beta}(s, M_d, L_d)$, the parameters $\beta$ and $L_d$ control the proportion of small entries in the support of $A$. 
If $\beta$ is large and $L_d$ grows slowly, then the fraction of weak signals in $A$ is small and therefore the false negatives (type II error) can also be well controlled. Below, we shall give such an example.

\begin{exmp}[A spatial design]
\label{exmp:spatial_A}
Let $0< r< 1$. Consider a $d\times d$ symmetric matrix $A=(A_{mk})_{d\times d}$ that is generated by the covariance function of a spatial process $Z_1,Z_2,...,Z_d$, which is a random vector that is observed at sites $h_1^\circ,...,h_d^\circ\in D$ in some spatial domain $D$. Assume that the covariance between $Z_m$ and $Z_k, m,k=1,\cdots,d$ satisfies
\begin{equation}
A_{mk} :=\text{cov}(Z_m,Z_k)=\left\{\begin{array}{ll}
f(\pi(h_m^\circ,h_k^\circ))&\pi(h_m^\circ,h_k^\circ)<\frac{d^{1-r}}{2}\\
0&\pi(h_m^\circ,h_k^\circ)\geq\frac{d^{1-r}}{2}
\end{array}
\right.,
\end{equation}
where $\pi(h_m^\circ,h_k^\circ)$ is the distance between the sites $h_m^\circ$ and $h_k^\circ$
\begin{align}
\pi(h_m^\circ,h_k^\circ)=\frac{|m-k|}{d^r}
\end{align}
and $f$ is a real-value covariance function. Here, we consider the rational quadratic covariance function \cite{chenxuwu2013,wikle2010general}
\begin{align}
f(\pi(h_m^\circ,h_k^\circ))=\frac{1}{(1+\pi(h_m^\circ,h_k^\circ)^2)^\gamma},\quad \gamma>1.
\end{align}
In this example, $A_{mk}=0$ if $|m-k| \ge d/2$, so $A$ is a banded matrix. 
\label{pattern_example}
\end{exmp}

\begin{lem}
\label{lem:spatial_A}
For $A$ in Example \ref{exmp:spatial_A}, then there exists a large enough constant $C > 0$ depending only on $r,\alpha,\gamma$ such that $A \in {\cal G}_{\alpha, \beta}(s, M_d, L_d)$, where
\begin{equation}
s=\left\{\begin{array}{lc}
Cd^r\log d& \text{if } 2\gamma\alpha=1\\
Cd^r& \text{if } 2\gamma\alpha>1\\
Cd^rd^{(1-r)(1-2\gamma\alpha)}&  \text{if } 0<2\gamma\alpha<1
\end{array}\right.,
\end{equation}
\begin{equation}
M_d=\left\{\begin{array}{lc}
Cd^r\log d& \text{if } \gamma=1/2\\
Cd^r& \text{if } \gamma>1/2\\
Cd^rd^{(1-r)(1-2\gamma)}& \text{if }  0<\gamma<1/2
\end{array}\right.,
\end{equation}
and $L_d=C d^{2(1-r)\gamma\beta}$. 
\end{lem}

For any fixed distance parameter $r \in (0,1)$, the weak signal parameter $L_d$ has a natural dependence on $\gamma$. If $\gamma$ is smaller, then the covariance function $f$ decays to zero slower and there is a less fraction of weak signals in $A$. This can allow $L_d$ to grow slowly in $d$. Note that class ${\cal G}_{\alpha, \beta}(s, M_d, L_d)$ is much less stringent than the widely used condition for support recovery and model selection in literature, which requires that the minimal nonzero signal strength is uniformly bounded away from zero \cite{ravikumarwainwrightraskuttiyu2008a}. To quantify the error in the pattern recovery, we use the following two error rate measures.
\begin{defn}
\label{defn:pattern-recovery-error-rate}
The false positive rate (FPR) and false negative rate (FNR) of $\hat{S}_{i}$ are defined as
\begin{equation}
\label{eqn:fpr+fnr}
\FPR_i = { |\hat{S}_{i} \cap S_{i}^c| \over |S_{i}^c| }, \quad \FNR_i = { |\hat{S}_{i}^c \cap S_{i}| \over |S_{i}| }.
\end{equation}
By convention, if $S_i^c=\emptyset$, then $\FPR_i=0$; if $S_i=\emptyset$, then $\FNR_i=0$.
\end{defn}
If $\FPR_i = \FNR_i = 0$ with probability tending to one, which is a very strong requirement, then we have the pattern recovery consistency $\Prob(\hat{S}_{i} = S_{i}) \to 1$. Below, we show that both FPR and FNR are asymptotically controlled in the presence of weak signals.

\begin{thm}
\label{eqn:pattern-recovery-consistency}
Assume that Assumption 1,2,3 are satisfied. Fix an $i \in [nb_n+1, n(1-b_n)+1]$. If $A_i \in {\cal G}_{\alpha}(s, M_d) $, then we have $\Prob(\FPR_i = 0) \ge 1 - 2 d^{-1}$. If in addtiion $A_{i,m} \in {\cal G}_{\alpha, \beta}(s, M_d, L_d)$, then we also have $\Prob(\FNR_i \le 2^\beta L_d u_\sharp^\beta) \ge 1 - 2 d^{-1}$.
\end{thm}

Since $u_\sharp = o(1)$, the FNR vanishes with probability tending to one if $L_d = o(u_\sharp^{-\beta})$.

\section{Simulation studies}
\label{sec:simulation}
In this section, we present some numerical results on simulated datasets. We compare the following five methods:
\begin{enumerate}[label=(\roman*)]
\item Our TV-VAR estimator (\ref{eqn:tvvar-clime}).
\item The stationary vector autoregressive model (Stat. VAR) \cite{hanluliu2015}.
\item The time-varying lasso method \cite{hsuhungchang2008a}.
\item The time-varying ridge method \cite{hamilton1994}.
\item The time-varying maximum likelihood estimator (MLE).
\end{enumerate}
Methods (iii), (iv) and (v) are extensions of the lasso \cite{hsuhungchang2008a}, ridge regression \cite{hamilton1994} and MLE to the time-varying setting by kernel smoothing. 
We include (ii), i.e. the stationary VAR model, as the baseline method that ignores the dynamic features in the transition matrices. (iii) is a competitor of (i). We solve (iii) with the FISTA \cite{beck2009fast} and we solve (i) and (ii) with the simplex algorithm \cite{murty1983linear}.

\subsection{Data generation}
\label{subsec:data}
We consider different setups of $n = 100$ and $d = 20, 30, 40$ and $50$. For each setup $(n,d)$, the data are generated by the following procedure. First, the baseline coefficient matrices $A_{01}$ and $A_{02}$ are generated by using the \texttt{sugm.generator()} in \texttt{flare} R package \cite{lizhaoyuanliu2014}.
We consider four graph structures defined in \texttt{flare}: hub, cluster, band and random for $A_{01}$ and $A_{02}$. Examples of these four structures are shown in Appendix \ref{four_sparse_patterns}. Then we normalize $\rho(A_{01}) = 0.2$, $\rho(A_{02}) = 1$ and smoothly interpolate on the intermediate values
\begin{equation*}
A_{i} =(1-t_i)^4 A_{01} + t_i^2 A_{02}, \qquad t_i = i/n \text{ for } i = 1, \cdots,n.
\end{equation*}
Following \cite{hanluliu2015}, we specify the innovation covariance matrix $\Psi = \Sigma - A_{01} \Sigma A_{01}^\top$ where $\Sigma = I_d$. \cite{hanluliu2015} showed that the choice of $\Sigma$ does not affect the numeric performance significantly. In our simulation studies, we use the Epanechnikov kernel $K(v) = 0.75(1-v^2)\I(|v| \le 1)$ and fix the model order $k=1$. The bandwidth $b_n$ is set to be $b_n=0.8n^{-1/5}$ for (i), (iii) and (iv).

\subsection{Tuning parameter selection}
\label{subsec:cv}
For the tuning parameter selection in (i)--(iv), we propose a data-driven procedure that minimizes the one-step-ahead prediction errors as follows.
\begin{enumerate}
\item Choose a grid for the tuning parameter (say $\tau$) and the number $n_1$ of training data points. 
\item For each $\tau$, perform the one-step-ahead prediction on the testing set by estimating $A_t$ with $\hat{A}_t(\tau)$ and then predict $X_t$ by $\hat{X}_t=\hat{A}_t(\tau)X_{t-1}$, where $t=n_1+1,...,n$. Then calculate the prediction error at time $t$ as $Err_t(\tau)=||X_t-\hat{A}_t(\tau)X_{t-1}||_2$.
\item Calculate the average errors over the last $n-n_1$ time points
$$\overline{Err(\tau)}=\frac{1}{n-n_1}\sum_{t=n_1+1}^{n}Err_t(\tau)$$
\item Select $\hat\tau$ that minimizes $\overline{Err(\tau)}$.
\end{enumerate}

\subsection{Comparison results}
\label{subsec:results}

\subsubsection{Estimation errors}
For each setup, we run 100 simulations and for each simulation indexed by $nn$ we estimate the transition matrix at each time point indexed by $t$, where $t=1,..,n$. Then, we calculate the error $err_{t,nn}=|\hat{A}_{t,nn}-A_t|_\rho$ for $t=1,...,n$ and $nn=1,...,100$. We also report errors under $\ell_F$, $\ell_1$ and $\ell_\infty$ norms. Next, the averaged errors over different time points are aggregated as $\overline{err}_{nn}=\frac{1}{b-a}\sum_{t=a}^berr_{t,nn}$ where $a=[nb_n]+1$ and $b=[n(1-b_n)]-1$. Finally, the averaged errors are reported over 100 simulations as $\overline{err}=\frac{1}{100}\sum_{nn=1}^{100}err_{nn}$.  The results are shown in Table \ref{tab:err_hub}, \ref{tab:err_cluster}, \ref{tab:err_band} and \ref{tab:err_random}.

From Table \ref{tab:err_hub} to Table \ref{tab:err_random}, larger $d$ often results in larger errors. In general, the unstructured MLE performs the worst almost under all matrix norms. Although the ridge method shrinks the coefficients in the transition matrix to zeros, the values of those coefficients are not exactly zeros. Thus, the estimated transition matrices by the ridge method are not sparse which lead to higher estimation errors compared with the TV-VAR,  time-varying lasso and stationary VAR. Stationary VAR always perform worse than TV-VAR and time-varying lasso, since it cannot capture the dynamic structure. The time-varying lasso performs better than any other methods except TV-VAR. The proposed TV-VAR performs the best almost under all matrix norms.

\subsubsection{Pattern recovery}
For the TV-VAR, we also report the $\FPR_{t,l,nn}$ and $\FNR_{t,l,nn}$ in (\ref{eqn:fpr+fnr}), where $l=1,..,30$ which are the indexes of 30 possible tuning parameters from 0.001 to 0.45. Then, we calculate the averaged $\FPR_l=\frac{1}{100}\sum_{nn=1}^{100}\{\frac{1}{b-a}\sum_{t=a}^b\FPR_{t,l,nn}\}$ and $\FNR_l=\frac{1}{100}\sum_{nn=1}^{100}\{\frac{1}{b-a}\sum_{t=a}^b\FNR_{t,l,nn}\}$. Following \cite{cailiuluo2011a}, we set the $u_\sharp$ to $10^{-3}$, which is considered to be numerical nonzero. The ROC curves for all possible values of the sparsity control parameter are plotted in Fig \ref{fig:roc_hub_fix_tau}, Fig \ref{fig:roc_cluster_fix_tau}, Fig \ref{fig:roc_band_fix_tau} and Fig \ref{fig:roc_random_fix_tau}. Based on the ROC curves, the TV-VAR method has better discrimination power for band or random patterns than hub or cluster patterns.

We also calculate $u_\sharp$ by using the true value of $\Sigma$ in (\ref{eqn:u_sharp}) and the resulting ROC curves are shown in Appendix \ref{App:true_ROC}. Those ROC curves are similar to those based on $u_\sharp = 10^{-3}$.

\begin{table}[htbp]
	 \centering
	 \caption{Comparison of estimation errors under different setups. The standard deviations are shown in parentheses. Here $\ell_\rho$, $\ell_F$, $\ell_1$ and $\ell_\infty$ are the spectral, Frobenius, $\ell_1$ and $\ell_\infty$ matrix norm resp. The pattern of transition matrix is `hub'.}
	\begin{tabular}{rrrrrr}
		\toprule
		& \multicolumn{5}{c}{d=20} \\
		\cmidrule(lr){2-6}
		  & TV-VAR    & Stat. VAR  & Lasso & Ridge & MLE \\ \midrule
		$\ell_\infty$   & \textbf{0.407} & 1.102 & 0.453 & 1.530 & 2.770 \\
		& (0.083) & (0.248) & (0.068) & (0.075) & (0.205) \\
		$\ell_1$    & \textbf{0.395} & 1.747 & 0.446 & 1.532 & 2.762 \\
		& (0.088) & (0.516) & (0.076) & (0.068) & (0.255) \\
		$\ell_\rho$  & \textbf{0.329} & 0.778 & 0.348 & 0.604 & 1.157 \\
		& (0.038) & (0.147) & (0.032) & (0.022) & (0.08) \\
		$\ell_F$     & 0.239 & 0.292 & \textbf{0.228} & 0.342 & 0.568 \\
		& (0.011) & (0.032) & (0.011) & (0.008) & (0.02) \\ \midrule
		& \multicolumn{5}{c}{d=30} \\ \cmidrule(lr){2-6}
		  & TV-VAR    & Stat. VAR  & Lasso & Ridge & MLE \\ \midrule
		$\ell_\infty$   & \textbf{0.643} & 1.154 & 0.695 & 2.116 & 4.350 \\
		& (0.125) & (0.295) & (0.111) & (0.087) & (0.302) \\
		$\ell_1$    & \textbf{0.694} & 2.197 & 0.732 & 2.102 & 4.302 \\
		& (0.135) & (0.652) & (0.109) & (0.073) & (0.302) \\
		$\ell_\rho$  & \textbf{0.430} & 0.846 & 0.443 & 0.710 & 1.607 \\
		& (0.046) & (0.141) & (0.04) & (0.024) & (0.112) \\
		$\ell_F$    & 0.245 & 0.288 & \textbf{0.244} & 0.393 & 0.729 \\
		& (0.013) & (0.026) & (0.014) & (0.006) & (0.02) \\ \midrule
		& \multicolumn{5}{c}{d=40} \\ \cmidrule(lr){2-6}
		  & TV-VAR    & Stat. VAR  & Lasso & Ridge & MLE \\ \midrule
		$\ell_\infty$   & \textbf{0.711} & 1.221 & 0.800 & 2.594 & 6.183 \\
		& (0.103) & (0.256) & (0.095) & (0.079) & (0.426) \\
		$\ell_1$    & \textbf{0.812} & 2.868 & 0.872 & 2.586 & 6.198 \\
		& (0.144) & (0.775) & (0.126) & (0.082) & (0.406) \\
		$\ell_\rho$  & \textbf{0.460} & 0.958 & 0.486 & 0.793 & 2.168 \\
		& (0.041) & (0.124) & (0.035) & (0.024) & (0.142) \\
		$\ell_F$     & \textbf{0.253} & 0.296 & 0.259 & 0.429 & 0.907 \\
		& (0.011) & (0.02) & (0.011) & (0.006) & (0.021) \\ \midrule
		& \multicolumn{5}{c}{d=50} \\ \cmidrule(lr){2-6}
		   & TV-VAR    & Stat. VAR  & Lasso & Ridge & MLE \\ \midrule
		$\ell_\infty$   & \textbf{0.827} & 1.158 & 0.944 & 3.043 & 8.478 \\
		& (0.124) & (0.219) & (0.102) & (0.095) & (0.633) \\
		$\ell_1$    & \textbf{1.019} & 3.574 & 1.072 & 3.055 & 8.558 \\
		& (0.19) & (1.021) & (0.135) & (0.098) & (0.65) \\
		$\ell_\rho$  & \textbf{0.504} & 1.052 & 0.531 & 0.863 & 2.874 \\
		& (0.045) & (0.157) & (0.032) & (0.021) & (0.224) \\
		$\ell_F$     & \textbf{0.266} & 0.297 & 0.278 & 0.454 & 1.108 \\
		& (0.014) & (0.019) & (0.012) & (0.006) & (0.027) \\
		\bottomrule
	\end{tabular}%
	\label{tab:err_hub}%
\end{table}%

\begin{table}[htbp]
	\centering
	\caption{Comparison of estimation errors under different setups. The standard deviations are shown in parentheses. Here $\ell_\rho$, $\ell_F$, $\ell_1$ and $\ell_\infty$ are the spectral, Frobenius, $\ell_1$ and $\ell_\infty$ matrix norm resp. The pattern of transition matrix is `cluster'.}
	\begin{tabular}{rrrrrr}
		\toprule
		& \multicolumn{5}{c}{d=20} \\\cmidrule(lr){2-6}
		                 & TV-VAR & Stat. VAR & Lasso & Ridge & MLE \\ \midrule
		$\ell_\infty$   & \textbf{0.400} & 1.093 & 0.465 & 1.527 & 2.736 \\
		& (0.073) & (0.300) & (0.084) & (0.081) & (0.202) \\
		$\ell_1$    & \textbf{0.383} & 1.695 & 0.448 & 1.527 & 2.731 \\
		& (0.074) & (0.506) & (0.077) & (0.083) & (0.251) \\
		$\ell_\rho$  & \textbf{0.324} & 0.774 & 0.352 & 0.601 & 1.133 \\
		& (0.032) & (0.143) & (0.033) & (0.025) & (0.084) \\
		$\ell_F$     & 0.239 & 0.291 & \textbf{0.228} & 0.341 & 0.562 \\
		& (0.011) & (0.03) & (0.011) & (0.008) & (0.019) \\ \midrule
		& \multicolumn{5}{c}{d=30} \\ \cmidrule(lr){2-6}
		                & TV-VAR & Stat. VAR & Lasso & Ridge & MLE \\ \midrule
		$\ell_\infty$   & \textbf{0.557} & 1.127 & 0.631 & 2.095 & 4.305 \\
		& (0.096) & (0.285) & (0.097) & (0.083) & (0.308) \\
		$\ell_1$    & \textbf{0.589} & 2.236 & 0.664 & 2.120 & 4.343 \\
		& (0.132) & (0.659) & (0.110) & (0.093) & (0.297) \\
		$\ell_\rho$  & \textbf{0.399} & 0.848 & 0.424 & 0.710 & 1.607 \\
		& (0.045) & (0.135) & (0.038) & (0.024) & (0.112) \\
		$\ell_F$     & 0.241 & 0.289 & \textbf{0.239} & 0.393 & 0.729 \\
		& (0.009) & (0.027) & (0.010) & (0.006) & (0.020) \\ \midrule
		& \multicolumn{5}{c}{d=40} \\ \cmidrule(lr){2-6}
		                      & TV-VAR & Stat. VAR & Lasso & Ridge & MLE \\ \midrule
		$\ell_\infty$   & \textbf{0.705} & 1.139 & 0.783 & 2.612 & 6.179 \\
		& (0.115) & (0.225) & (0.108) & (0.106) & (0.44) \\
		$\ell_1$    & \textbf{0.803} & 2.986 & 0.847 & 2.596 & 6.171 \\
		& (0.153) & (0.874) & (0.133) & (0.098) & (0.452) \\
		$\ell_\rho$  & \textbf{0.460} & 0.974 & 0.479 & 0.794 & 2.158 \\
		& (0.042) & (0.15) & (0.037) & (0.022) & (0.156) \\
		$\ell_F$     & \textbf{0.253} & 0.295 & 0.258 & 0.428 & 0.906 \\
		& (0.011) & (0.022) & (0.012) & (0.006) & (0.022) \\ \midrule
		& \multicolumn{5}{c}{d=50} \\ \cmidrule(lr){2-6}
		                      & TV-VAR & Stat. VAR & Lasso & Ridge & MLE \\ \midrule
		$\ell_\infty$   & \textbf{0.841} & 1.116 & 0.964 & 3.052 & 8.593 \\
		& (0.121) & (0.198) & (0.108) & (0.106) & (0.703) \\
		$\ell_1$    & \textbf{1.023} & 3.384 & 1.070 & 3.043 & 8.498 \\
		& (0.177) & (1.055) & (0.152) & (0.09) & (0.558) \\
		$\ell_\rho$  & \textbf{0.504} & 1.015 & 0.531 & 0.863 & 2.874 \\
		& (0.04) & (0.162) & (0.036) & (0.021) & (0.224) \\
		$\ell_F$     & \textbf{0.267} & 0.292 & 0.278 & 0.454 & 1.108 \\
		& (0.013) & (0.019) & (0.013) & (0.006) & (0.027) \\
		\bottomrule
	\end{tabular}%
	\label{tab:err_cluster}%
\end{table}%

\begin{table}[htbp]
	\centering
	\caption{Comparison of estimation errors under different setups. The standard deviations are shown in parentheses. Here $\ell_\rho$, $\ell_F$, $\ell_1$ and $\ell_\infty$ are the spectral, Frobenius, $\ell_1$ and $\ell_\infty$ matrix norm resp. The pattern of transition matrix is `band'.}
	\begin{tabular}{rrrrrr}
		\toprule
		& \multicolumn{5}{c}{d=20} \\\cmidrule(lr){2-6}
		& TV-VAR & Stat. VAR & Lasso & Ridge & MLE \\\midrule
		$\ell_\infty$   & \textbf{0.402} & 1.048 & 0.453 & 1.532 & 2.738 \\
		& (0.072) & (0.251) & (0.068) & (0.08) & (0.188) \\
		$\ell_1$    & \textbf{0.385} & 1.645 & 0.437 & 1.531 & 2.707 \\
		& (0.075) & (0.476) & (0.077) & (0.068) & (0.183) \\
		$\ell_\rho$  & \textbf{0.326} & 0.758 & 0.347 & 0.600 & 1.132 \\
		& (0.033) & (0.137) & (0.032) & (0.021) & (0.072) \\
		$\ell_F$     & 0.237 & 0.289 & \textbf{0.226} & 0.341 & 0.563 \\
		& (0.011) & (0.031) & (0.01) & (0.008) & (0.019) \\\midrule
		& \multicolumn{5}{c}{d=30} \\\cmidrule(lr){2-6}
		& TV-VAR & Stat. VAR & Lasso & Ridge & MLE \\\midrule
		$\ell_\infty$  & \textbf{0.574} & 1.139 & 0.637 & 1.537 & 4.316 \\
		& (0.094) & (0.319) & (0.101) & (0.058) & (0.314) \\
		$\ell_1$    & \textbf{0.590} & 2.306 & 0.654 & 1.537 & 4.312 \\
		& (0.131) & (0.727) & (0.123) & (0.057) & (0.323) \\
		$\ell_\rho$  & \textbf{0.401} & 0.863 & 0.422 & 0.537 & 1.607 \\
		& (0.041) & (0.154) & (0.038) & (0.019) & (0.112) \\
		$\ell_F$     & 0.240 & 0.290 & \textbf{0.239} & 0.315 & 0.729 \\
		& (0.009) & (0.028) & (0.011) & (0.005) & (0.02) \\\midrule
		& \multicolumn{5}{c}{d=40} \\\cmidrule(lr){2-6}
		& TV-VAR & Stat. VAR & Lasso & Ridge & MLE \\\midrule
		$\ell_\infty$  & \textbf{0.800} & 1.170 & 0.886 & 2.590 & 6.198 \\
		& (0.134) & (0.253) & (0.113) & (0.098) & (0.429) \\
		$\ell_1$   & \textbf{0.925} & 2.772 & 0.964 & 2.597 & 6.221 \\
		& (0.155) & (0.825) & (0.127) & (0.09) & (0.457) \\
		$\ell_\rho$  & \textbf{0.489} & 0.938 & 0.510 & 0.793 & 2.168 \\
		& (0.041) & (0.138) & (0.035) & (0.024) & (0.142) \\
		$\ell_F$     & \textbf{0.262} & 0.293 & 0.270 & 0.429 & 0.907 \\
		& (0.013) & (0.022) & (0.012) & (0.006) & (0.021) \\\midrule
		& \multicolumn{5}{c}{d=50} \\\cmidrule(lr){2-6}
		& TV-VAR & Stat. VAR & Lasso & Ridge & MLE \\\midrule
		$\ell_\infty$   & \textbf{0.839} & 1.143 & 0.992 & 3.056 & 8.550 \\
		& (0.125) & (0.222) & (0.118) & (0.100) & (0.579) \\
		$\ell_1$    & \textbf{1.011} & 3.673 & 1.105 & 3.047 & 8.635 \\
		& (0.175) & (1.212) & (0.126) & (0.099) & (0.678) \\
		$\ell_\rho$  & \textbf{0.506} & 1.047 & 0.541 & 0.862 & 2.872 \\
		& (0.045) & (0.176) & (0.034) & (0.021) & (0.223) \\
		$\ell_F$     & \textbf{0.267} & 0.295 & 0.280 & 0.454 & 1.108 \\
		& (0.014) & (0.02) & (0.013) & (0.005) & (0.027) \\
		\bottomrule
	\end{tabular}%
	\label{tab:err_band}%
\end{table}%

\begin{table}[htbp]
	\centering
	\caption{Comparison of estimation errors under different setups. The standard deviations are shown in parentheses. Here $\ell_\rho$, $\ell_F$, $\ell_1$ and $\ell_\infty$ are the spectral, Frobenius, $\ell_1$ and $\ell_\infty$ matrix norm resp. The pattern of transition matrix is `random'.}
	\begin{tabular}{rrrrrr}
		\toprule
		& \multicolumn{5}{c}{d=20} \\\cmidrule(lr){2-6}
		& TV-VAR & Stat. VAR & Lasso & Ridge & MLE \\\midrule
		$\ell_\infty$   & \textbf{0.397} & 1.060 & 0.447 & 1.538 & 2.704 \\
		& (0.078) & (0.244) & (0.07) & (0.071) & (0.188) \\
		$\ell_1$    & \textbf{0.379} & 1.661 & 0.433 & 1.532 & 2.762 \\
		& (0.080) & (0.574) & (0.081) & (0.077) & (0.235) \\
		$\ell_\rho$  & \textbf{0.323} & 0.757 & 0.344 & 0.600 & 1.132 \\
		& (0.036) & (0.140) & (0.033) & (0.021) & (0.072) \\
		$\ell_F$     & 0.237 & 0.286 & \textbf{0.227} & 0.341 & 0.563 \\
		& (0.013) & (0.031) & (0.012) & (0.008) & (0.019) \\\midrule
		& \multicolumn{5}{c}{d=30} \\\cmidrule(lr){2-6}
		& TV-VAR & Stat. VAR & Lasso & Ridge & MLE \\\midrule
		$\ell_\infty$   & \textbf{0.567} & 1.138 & 0.641 & 1.541 & 4.294 \\
		& (0.103) & (0.264) & (0.102) & (0.066) & (0.292) \\
		$\ell_1$    & \textbf{0.590} & 2.280 & 0.658 & 1.543 & 4.292 \\
		& (0.132) & (0.724) & (0.119) & (0.062) & (0.268) \\
		$\ell_\rho$  & \textbf{0.398} & 0.865 & 0.423 & 0.537 & 1.607 \\
		& (0.044) & (0.146) & (0.041) & (0.019) & (0.112) \\
		$\ell_F$     & 0.240 & 0.290 & \textbf{0.239} & 0.315 & 0.729 \\
		& (0.01) & (0.026) & (0.012) & (0.005) & (0.020) \\\midrule
		& \multicolumn{5}{c}{d=40} \\\cmidrule(lr){2-6}
		& TV-VAR & Stat. VAR & Lasso & Ridge & MLE \\\midrule
		$\ell_\infty$   & \textbf{0.713} & 1.210 & 0.810 & 2.595 & 6.148 \\
		& (0.098) & (0.268) & (0.104) & (0.096) & (0.413) \\
		$\ell_1$    & \textbf{0.813} & 2.904 & 0.869 & 2.600 & 6.205 \\
		& (0.136) & (0.851) & (0.128) & (0.09) & (0.432) \\
		$\ell_\rho$  & \textbf{0.458} & 0.951 & 0.482 & 0.793 & 2.168 \\
		& (0.036) & (0.142) & (0.034) & (0.024) & (0.142) \\
		$\ell_F$     & \textbf{0.252} & 0.296 & 0.260 & 0.429 & 0.907 \\
		& (0.01) & (0.021) & (0.012) & (0.006) & (0.021) \\\midrule
		& \multicolumn{5}{c}{d=50} \\\cmidrule(lr){2-6}
		& TV-VAR & Stat. VAR & Lasso & Ridge & MLE \\\midrule
		$\ell_\infty$   & \textbf{0.830} & 1.207 & 1.002 & 3.060 & 8.452 \\
		& (0.109) & (0.234) & (0.120) & (0.108) & (0.553) \\
		$\ell_1$    & \textbf{1.023} & 3.529 & 1.129 & 3.061 & 8.457 \\
		& (0.187) & (1.214) & (0.151) & (0.092) & (0.551) \\
		$\ell_\rho$  & \textbf{0.504} & 1.034 & 0.543 & 0.865 & 2.861 \\
		& (0.039) & (0.180) & (0.033) & (0.022) & (0.198) \\
		$\ell_F$     & \textbf{0.267} & 0.298 & 0.282 & 0.455 & 1.107 \\
		& (0.013) & (0.021) & (0.013) & (0.005) & (0.027) \\
		\bottomrule
	\end{tabular}%
	\label{tab:err_random}%
\end{table}%

\begin{figure}[htp]
  \centering
  \label{figur}\caption{ROC curves under different settings}
  \subfloat[Pattern is hub]{\label{fig:roc_hub_fix_tau}\includegraphics[width=60mm]{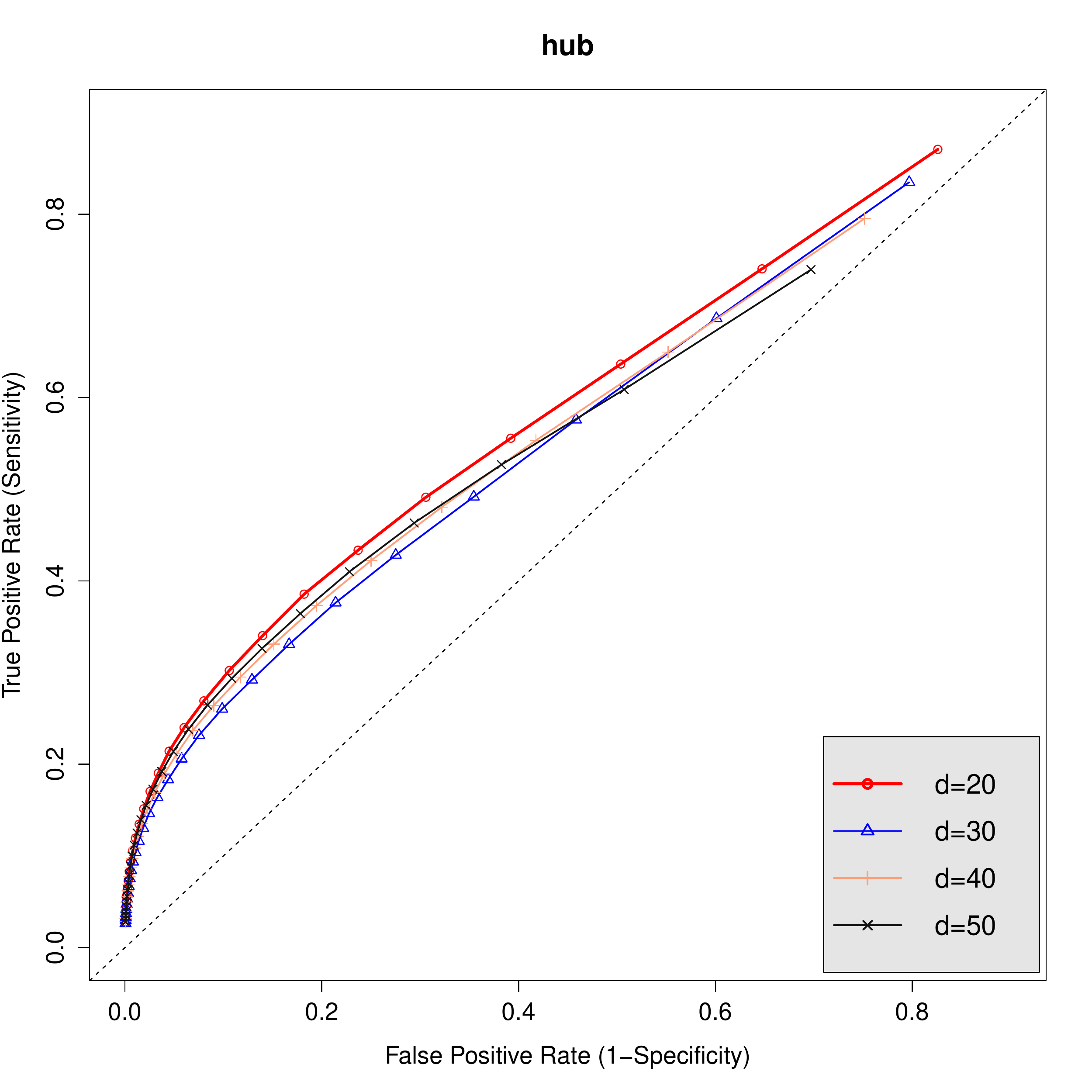}}
  \subfloat[Pattern is cluster]{\label{fig:roc_cluster_fix_tau}\includegraphics[width=60mm]{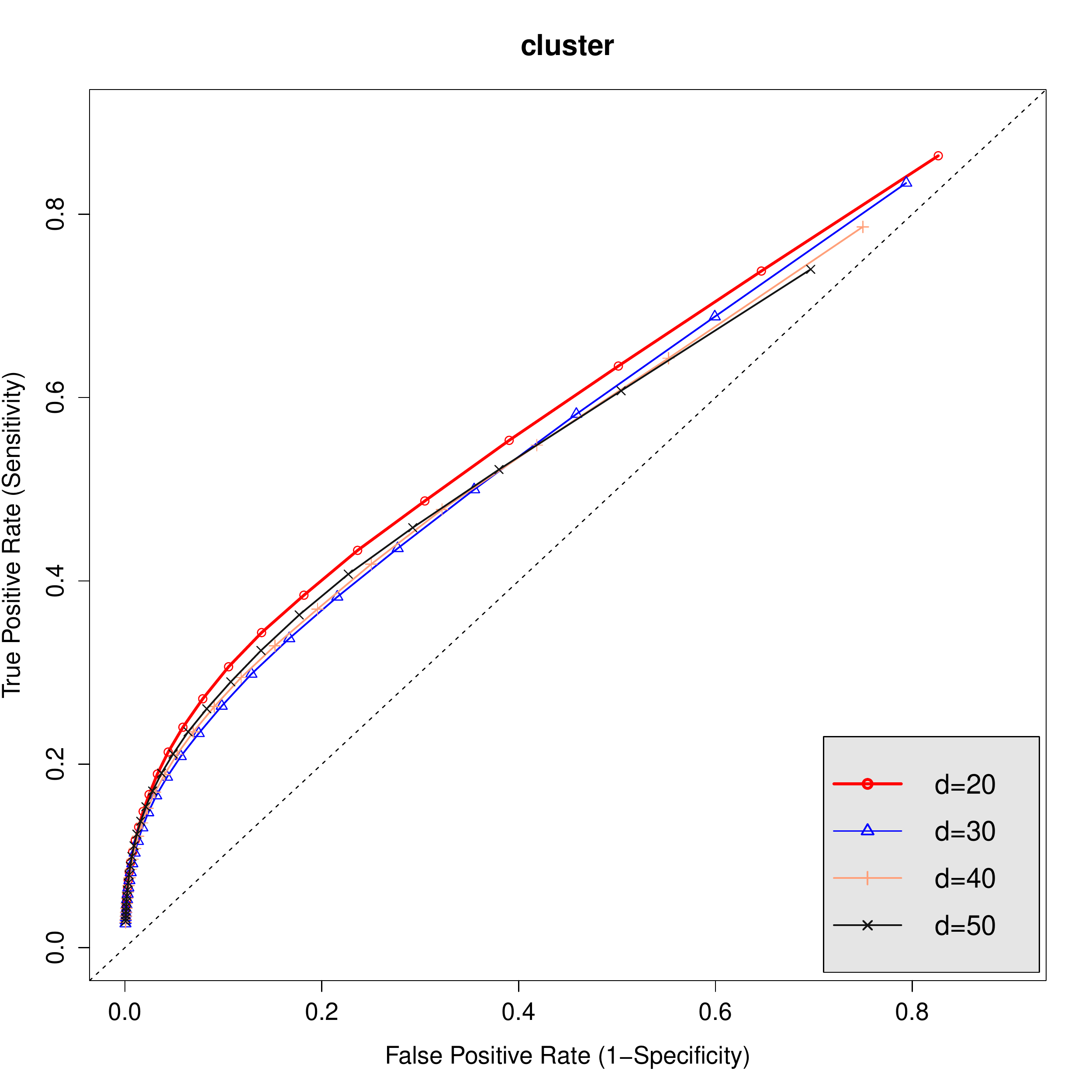}}
  \\
  \subfloat[Pattern is band]{\label{fig:roc_band_fix_tau}\includegraphics[width=60mm]{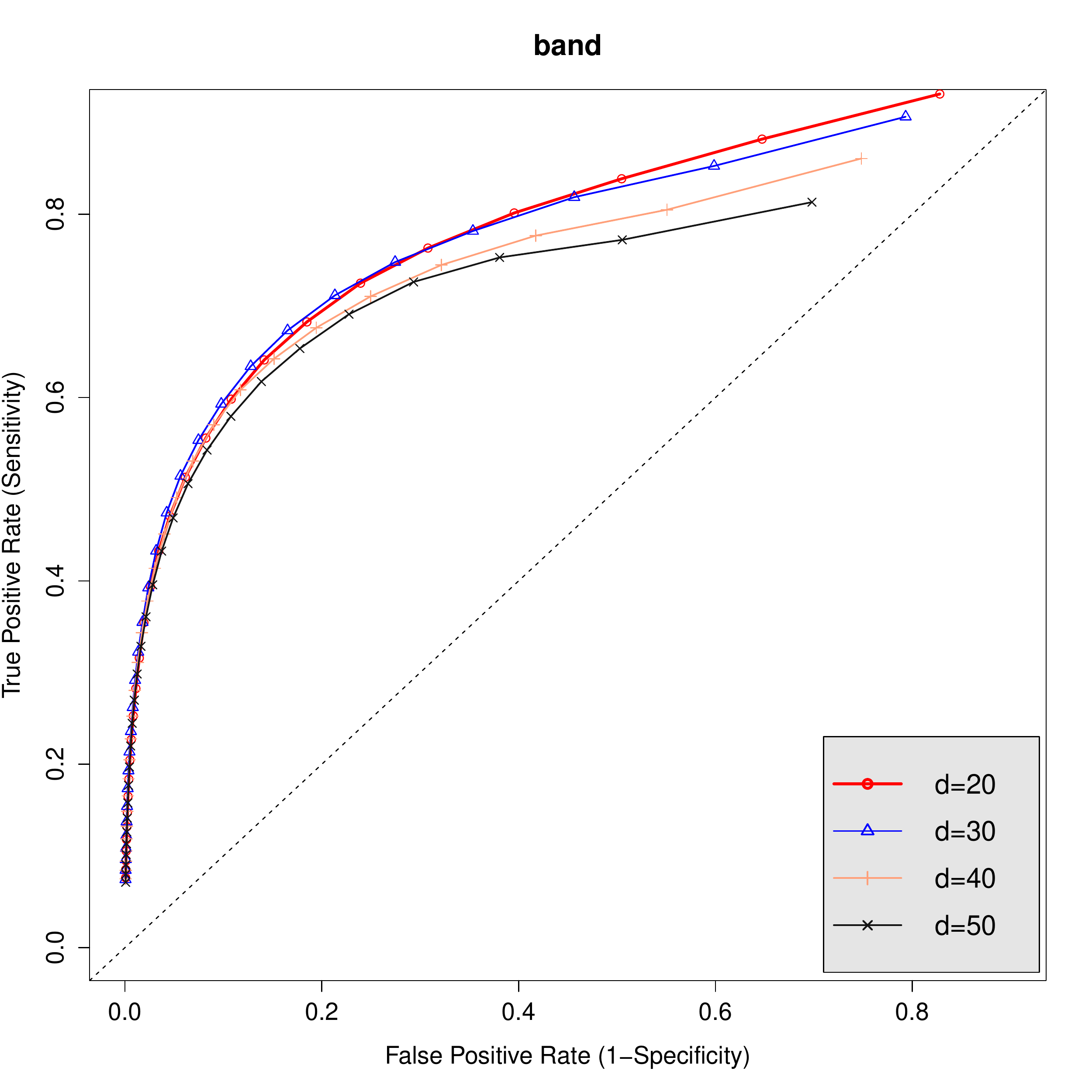}}
  \subfloat[Pattern is random]{\label{fig:roc_random_fix_tau}\includegraphics[width=60mm]{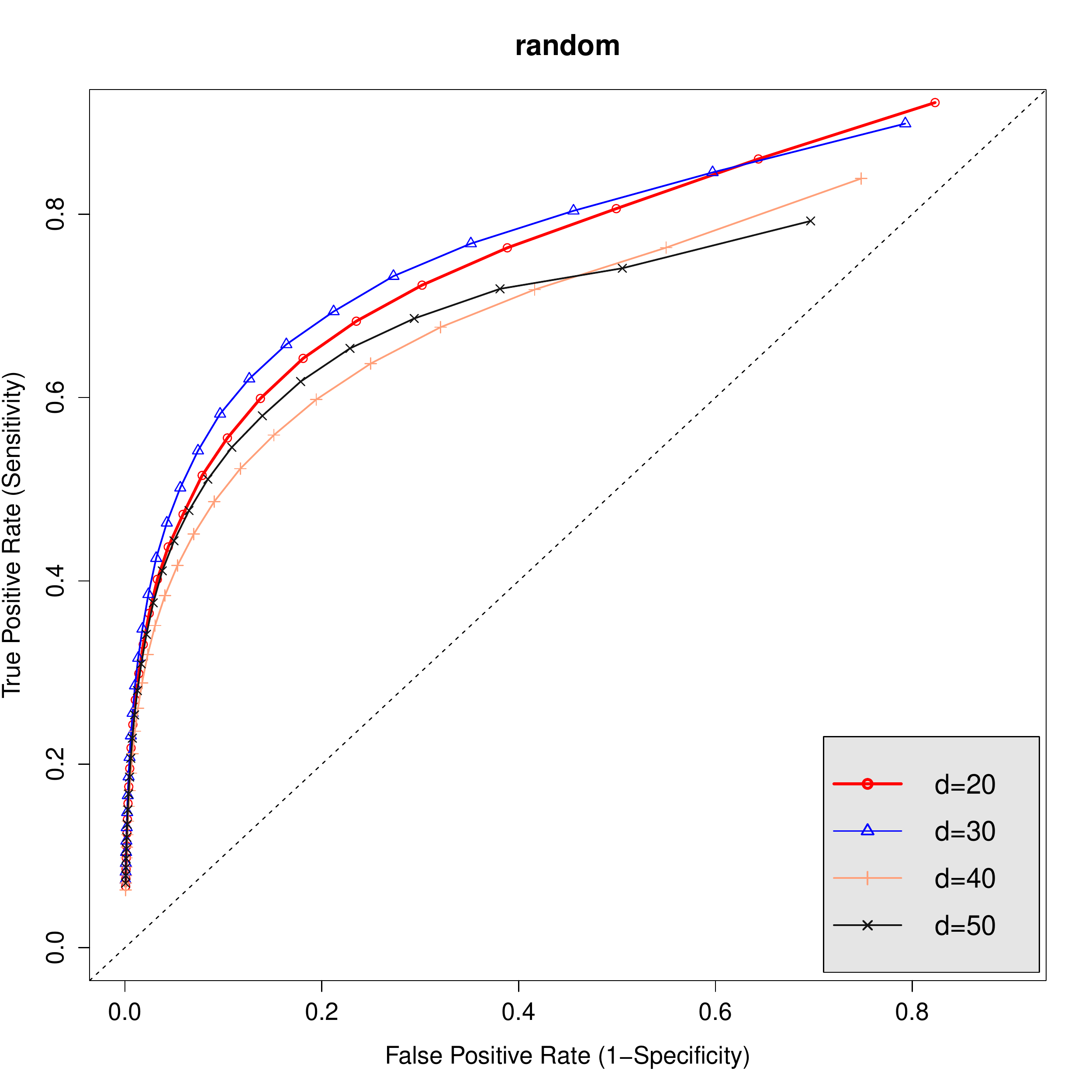}}
\end{figure}

\section{Real data analysis}
\label{sec:real-data-analysis}

\subsection{Finance data: stock prices}
\label{subsec:finance-data}
In this section, we compare the aforementioned estimators on a real financial dataset. The dataset is from Yahoo! Finance (\texttt{finance.yahoo.com}). The data matrix contains daily closing prices of 452 stocks that are consistently in the S\&P 500 index between January 1st, 2003 and January 1st, 2008. We choose such time range to avoid the effects of the two financial crises in the year of 2001 and 2008, which could make the stock prices non-smooth with sharp drops and pick-ups. In total, there are 1,258 time points. We first standardize the data to zero mean and unit variance and detrend the data. We then fit the detrended data for those stocks which are most smoothed (without obvious change points) using AR(1) model and perform the Ljung-Box Tests by using \texttt{Box.test()} in R. We keep those stocks with nonzero coefficients at significant level 0.05. 30 stocks are finally selected and 10 of them are shown in Table \ref{tab:finance_selsto}. The ten selected stocks in Table \ref{tab:finance_selsto} are from six sectors including: Consumer Staples, Consumer Discretionary, Industrials, Financials, Utilities and Energy. The resulting data matrix is denoted as $\bar{X}\in \mathbb{R}^{1258\times 30}$. We use the Priestley-Subba Rao (PSR) stationarity test (such as \texttt{stationarity()} in R package \texttt{fractal}) to some of the stocks such as Kellogg Co. and Target Corp. to find that these time series are not stationary at significant level $0.05$. Thus, it is inappropriate to model the data with the stationary VAR model  in \cite{hanluliu2015}. Finally we fit the data $\bar{X}$ into the sparse TV-VAR model with order $k=1$, stationary VAR \cite{hanluliu2015} with order $k=1$, time-varying lasso method \cite{hsuhungchang2008a},  time-varying ridge method \cite{hamilton1994} and time-varying MLE. 

In order to compare the performance of the five methods, we consider the one-step-ahead prediction for $\bar{X}_t$, $t=1159,...,1258$, by using $\bar{X}_{J_t,*}$, where $J_t=\{j:t-1158\leq j\leq t-1 \}$ as the training set. We estimate the transition matrices $\hat{A}_t(\lambda)$ where $\lambda$ is the regularization parameter such as the $\tau$ in the TV-VAR or the shrinkage parameter in the ridge method. We use the Epanechnikov kernel $K(v) = 0.75(1-v^2)\I(|v| \le 1)$ with bandwidth $b_n=0.3$, which means that only the last 30\% data in the training set are used for prediction for the TV-VAR, time-varying lasso method, time-varying ridge method and time-varying MLE. Since stationary VAR \cite{hanluliu2015} model uses all training data, in order to make fair comparisons, we only treat the last 30\% time points (347 observations) in $\bar{X}_{J_t,*}$ as the training set for the stationary VAR. The averaged prediction error for a specific $\lambda$ is measured by:
$$\overline{Err}(\lambda)=\frac{1}{100}\sum_{t=1159}^{1258}\|\bar{X}_{t,*}-\hat{A}_t(\lambda)\bar{X}_{t-1,*}\|_2.$$


The smallest averaged one-step-ahead prediction errors $min_{\lambda*}\overline{Err}(\lambda)$ for five methods were shown in Table \ref{tab:finance_1stb}. In Figure \ref{fig:fin_b1}, we show an example for the predicted closing prices from the five methods and the detrended true closing prices of Target Corp. Clearly, the methods with time-varying structures perform similarly but outperform stationary VAR.

One advantage of the TV-VAR compared with the stationary VAR is that it can capture the time-varying data structures. If we treat a transition matrix as an adjacency matrix for an undirected weighted graph, we can visualize the graph structures at different time points. For instance, Fig \ref{fig:fin_adja_1}, \ref{fig:fin_adja_2}, \ref{fig:fin_adja_3} and \ref{fig:fin_adja_4} are four graphical representations of the TV-VAR, which clearly show some time-evolving patterns. In the graphs, different vertices represent different stocks and the edges represent the cross-sectional dependence between stocks. Bolder lines indicate stronger dependence and the stocks in the same sector are in the same color. In order to illustrate the time-varying structure clearly, we only consider the 10 stocks in Table \ref{tab:finance_selsto}. From these four figures, we observe that stocks in the same and closely related sectors are often connected. For examples, CME Group Inc. and Hartford Financial Svc. GP from Financials sector show a dynamic dependence structure, and Boeing Company in Industrials shows the consistent correlation with Exxon Mobil Corp. in Energy sector.

\begin{table}[htbp]
  \centering
  \caption{The prediction errors for the five methods on the stock price data. The standard deviations are shown in parentheses.}
    \begin{tabular}{rccccc}
    \toprule
          & TV-VAR & Stat. VAR & Lasso & Ridge & MLE \\
    \midrule
   $\overline{Err}(\lambda*)$  & 0.4822 & 2.2824 & 0.4580 & 0.4843 & 0.4902 \\
    Standard Deviation & (0.1751) & (0.5169) & (0.1810) & (0.1747) & (0.1799) \\
    \bottomrule
    \end{tabular}%
  \label{tab:finance_1stb}%
\end{table}%

\begin{figure}[htbp]
\centering
 \includegraphics[width=11cm,height=7.4cm]{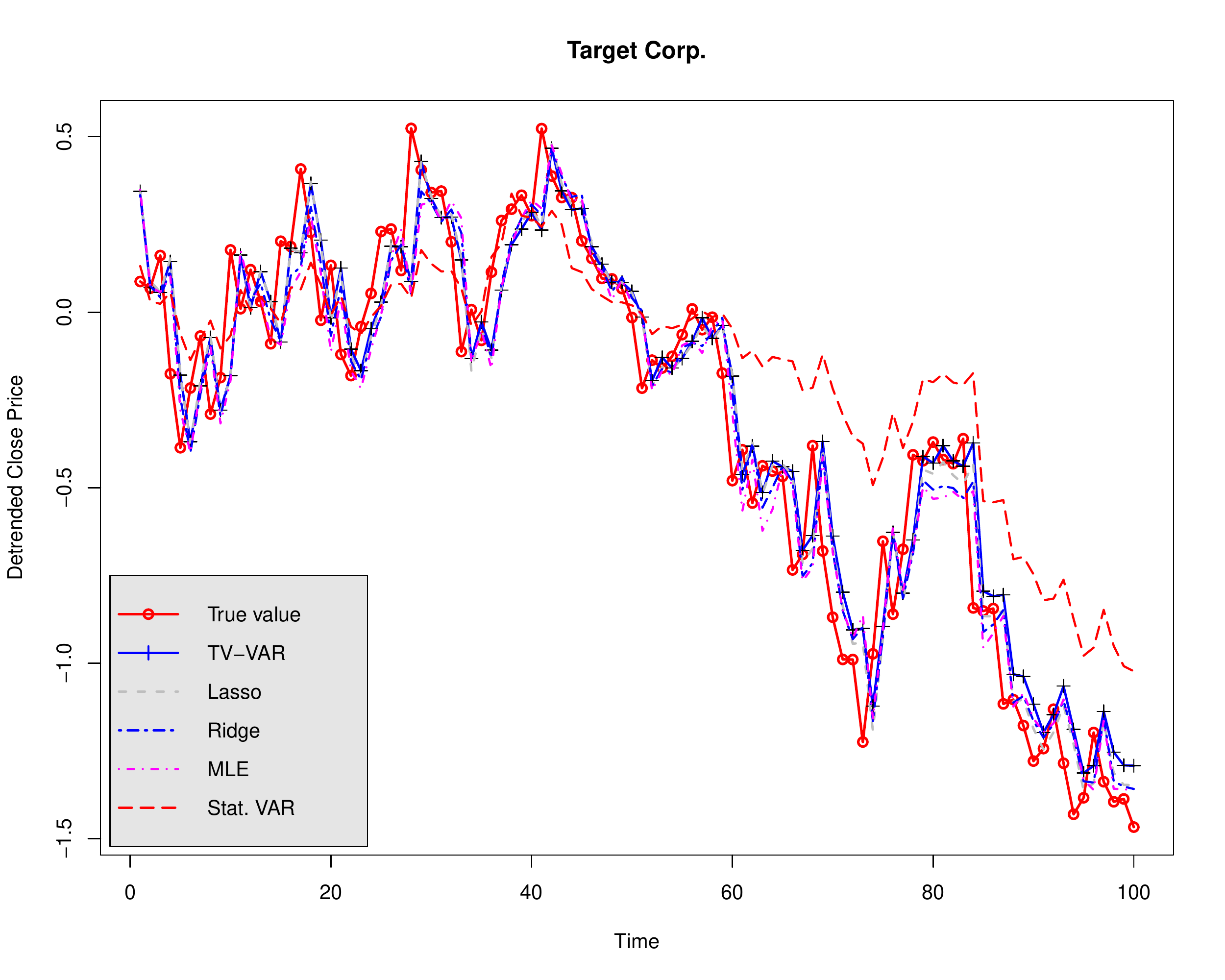}
 \caption{Comparison of the predicted closing prices and the true detrended closing prices of Target Corp.}
 \label{fig:fin_b1}
\end{figure}

\begin{figure}[htbp]
  \centering
  \label{figur}\caption{The graphs based on adjacency matrices estimated by TV-VAR}
  \subfloat[Adjacency matrix $\hat{A}_{1183}$]{\label{fig:fin_adja_1}\includegraphics[width=60mm]{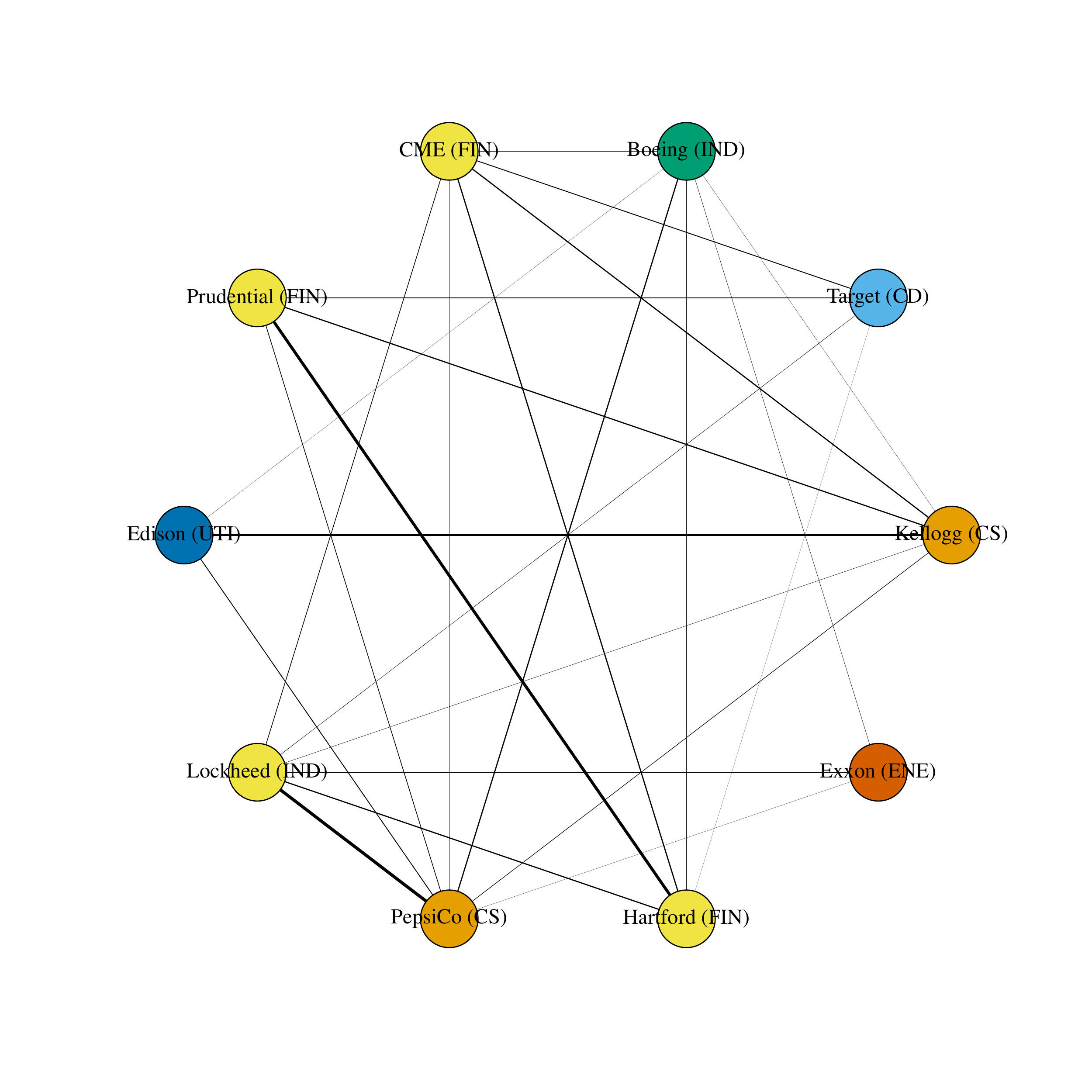}}
  \subfloat[Adjacency matrix $\hat{A}_{1208}$]{\label{fig:fin_adja_2}\includegraphics[width=60mm]{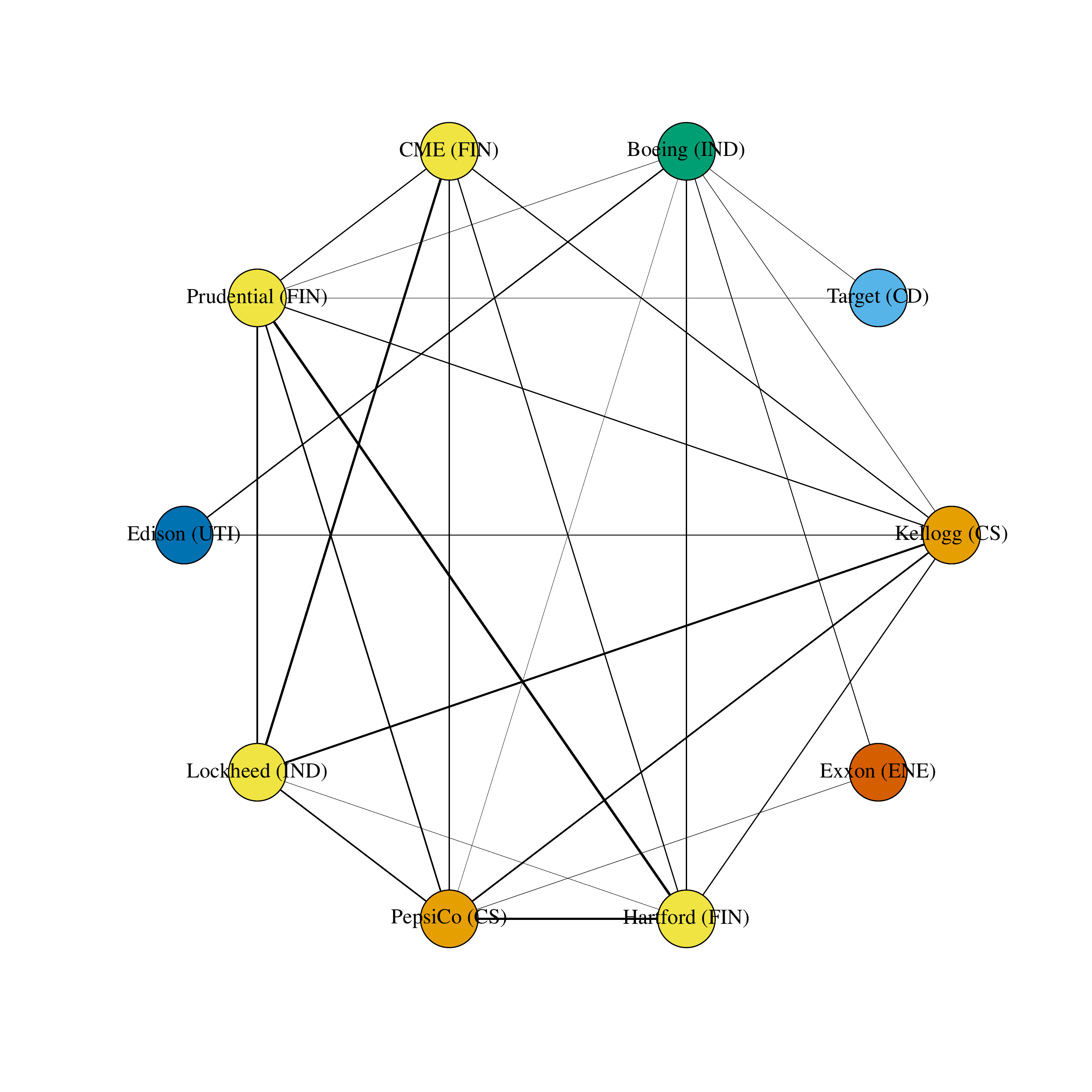}}
  \\
  \subfloat[Adjacency matrix $\hat{A}_{1233}$]{\label{fig:fin_adja_3}\includegraphics[width=60mm]{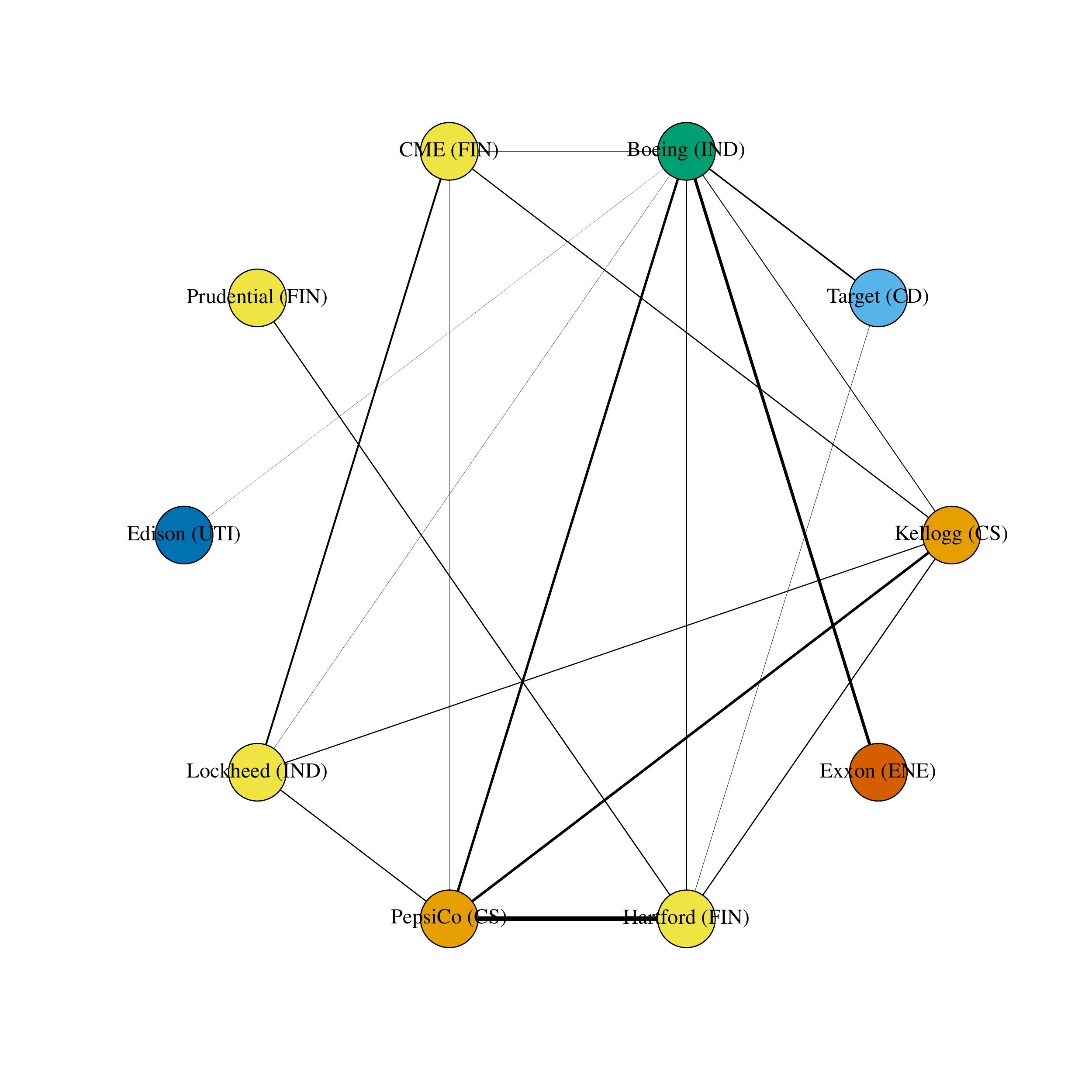}}
  \subfloat[Adjacency matrix $\hat{A}_{1258}$]{\label{fig:fin_adja_4}\includegraphics[width=60mm]{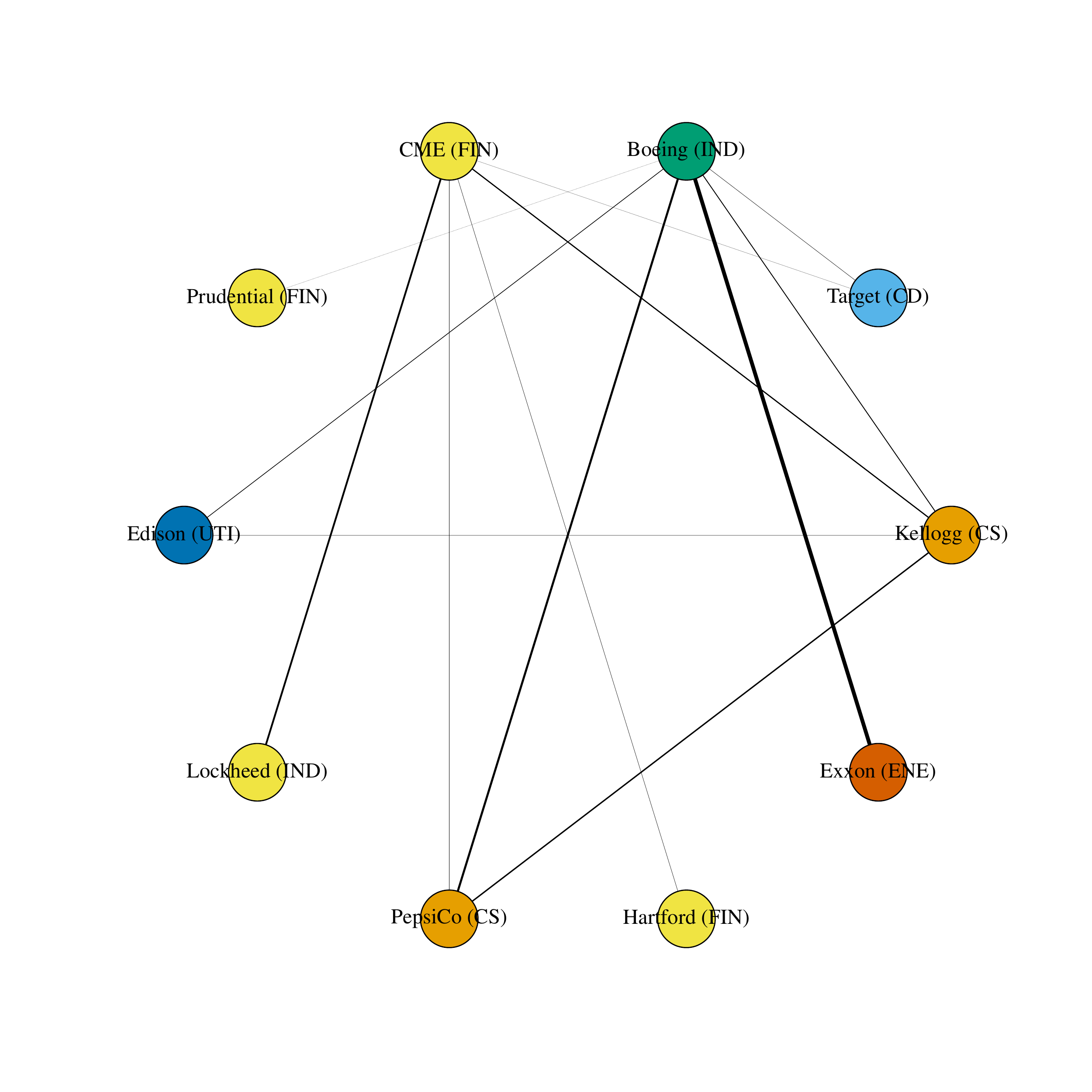}}
\end{figure}

\subsection{Economic data: exchange rates}
\label{subsec:economic-data}

We also apply the exchange rates data on international currencies and study how the exchange rates evolve over time. We use the exchange rates data from the Federal Reserve Bank at St. Louis and choose the date range from January 1st, 2003 to January 1st, 2008 to avoid the financial crises' effects on exchange rates. In total, there are 1,303 time points. We choose 15 major currencies from the continents including Europe, Australia, North America, Asian, South America and we normalize each currency by the exchange rate of the U.S. dollar shown in Table \ref{tab:exchange_selcur}. As in Section \ref{subsec:finance-data}, we standardize and detrend the data. The exchange rates for all currencies to the U.S. dollars are smoothed and in addition the PSR test rejects the stationarity hypothesis of the exchange rate Euro/U.S. at the significant level $0.05$. Thus, stationary VAR is also not appropriate in this application, which is also supported by our empirical findings in Table \ref{tab:exchange_1stb} and Figure \ref{fig:exch_b1}. We compare the performance of the five methods by the one-step-ahead prediction errors for the last 50 data and use the same prediction metric in Section \ref{subsec:finance-data}. The bandwidth here is $b_n=0.3$. We do not include the prediction curve of the stationary VAR in Figure \ref{fig:exch_b1} since its errors are too large. Table \ref{tab:exchange_1stb} and Figure \ref{fig:exch_b1} show that  methods with time-varying structure performs similarly but much better than the stationary VAR.

\begin{table}[htbp]
  \centering
  \caption{The prediction errors for the five methods on the exchange rates data. The standard deviations are shown in parentheses.}
    \begin{tabular}{rccccc}
    \toprule
          & TV-VAR & Stat. Var & Lasso & Ridge & MLE \\
    \midrule
     $\overline{Err}(\lambda*)$    & 0.2990 & 8.7637 & 0.2554 & 0.2634 & 0.2638 \\
    Standard Deviation & (0.0924) & (5.2835) & (0.1037)  & (0.1032)& (0.1032) \\
    \bottomrule
    \end{tabular}%
  \label{tab:exchange_1stb}%
\end{table}%

\begin{figure}[htbp]
\centering
 \includegraphics[width=12cm,height=8cm]{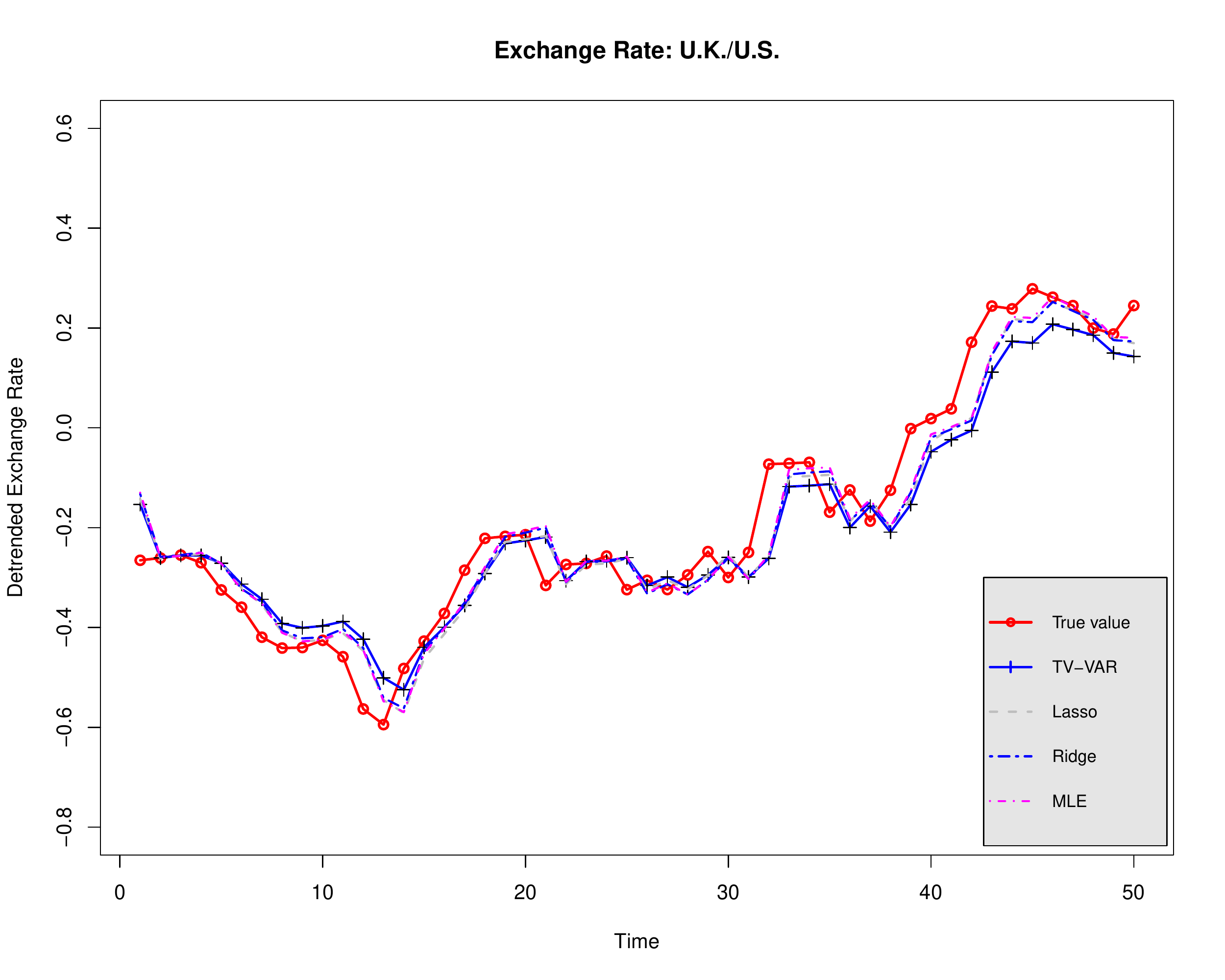}
 \caption{Compare the predicted and true detrended exchange rate of U.K./U.S.}
 \label{fig:exch_b1}
\end{figure}

\section{Proofs}
\label{sec:proofs}

\begin{proof}[\textbf{Proof of Lemma \ref{lem:smoothness-relation}}]
Since $\tilde\vx_m(t)$ is a stationary VAR, we have $\tilde\Sigma(t) =I_{d \times d}+ A(t) \tilde\Sigma(t)A(t)^\top $. By the extension, $\Sigma(t)=\tilde\Sigma(t)$ for all $t \in (0,1)$. Applying the chain rule of differentiation w.r.t. $t$, we get
$$
\dot\Sigma(t) = \dot{A}(t) \Sigma(t) A(t)^\top + A(t) \dot\Sigma(t) A(t)^\top + A(t) \Sigma(t) \dot{A}(t)^\top.
$$
By Lemma \ref{lem:norm-bounds-three-matrices} and the triangle inequality,
\begin{align*}
&|\dot\Sigma(t)|_\infty \\
\le &|\dot{A}(t)|_{\infty} |\Sigma(t)|_{\ell^\infty} |A(t)|_{\ell^1} + |A(t)|_{\ell^1} |\dot\Sigma(t)|_\infty |A(t)|_{\ell^1} + |A(t)|_{\ell^1} |\Sigma(t)|_{\ell^1} |\dot{A}(t)|_{\infty}.
\end{align*}

Then, (\ref{eqn:smoothness-relation-1dev}) follows from the assumption that $\sup_{t \in [0,1]} |A(t)|_{\ell^1} < 1$ and that $\Sigma(t)$ is symmetric. (\ref{eqn:smoothness-relation-2dev}) follows from the differentiation of $\dot\Sigma(t)$ w.r.t. $t$.
\end{proof}

\begin{lem}
\label{lem:norm-bounds-three-matrices}
Let $A,B,C$ be matrices of compatible dimension for the product $ABC$. Then
\begin{eqnarray*}
|ABC|_\infty &\le& |A|_{\ell^1} |B|_\infty |C|_{\ell^\infty}, \\
|ABC|_\infty &\le& |A|_{\ell^1} |B|_{\ell^1} |C|_{\infty}, \\
|ABC|_\infty &\le& |A|_{\infty} |B|_{\ell^\infty} |C|_{\ell^\infty}.
\end{eqnarray*}
\end{lem}

\begin{proof}[\textbf{Proof of Lemma \ref{lem:norm-bounds-three-matrices}}]
The first and second inequalities follow from
\begin{eqnarray*}
|ABC|_\infty &=& \max_{i,j} \left| \sum_{k,m} A_{ik}B_{km}C_{mj} \right| \le \max_{i,j} \sum_k |A_{ik}| \sum_{m} |B_{km}C_{mj} | \\
&\le& \left(\max_i \sum_k |A_{ik}| \right) \max_{j,k} \sum_m |B_{km}C_{mj}| \\
&\le& |A|_{\ell^1} \min\left\{ |B|_\infty |C|_{\ell^\infty}, \; |B|_{\ell^1} |C|_\infty \right\}.
\end{eqnarray*}
The third inequality follows from the second one by considering $(ABC)^\top = C^\top B^\top A^\top$ and $|A^\top|_{\ell^1} = |A|_{\ell^\infty}$.
\end{proof}

Recall the TV-VAR(1) model $\vx_i = A_i \vx_{i-1} + \ve_i$. The following key lemma presents a large deviation bound for the marginal and lag-one autocovariance matrices with sub-Gaussian innovations.
\vspace{0.1in}
\begin{lem}
\label{lem:large-dev-Sigma0}
Suppose $\bar{\rho} = \sup_{i \ge 0} \rho(A_i) < 1$ and $$\sup_{t \in [0,1]}  |\dot\Sigma_{m,jk}(t)| \vee |\ddot\Sigma_{m,jk}(t)| \le C_0, m = 0,1,$$ for some absolute constant $C_0 < \infty$. If $\ve_i = (e_{i1}, \cdots, e_{id})$ has iid sub-Gaussian components, i.e. $\|e_{ij}\|_q \le C_1 q^{1/2}$ for all $q \ge 1$, then, for any fixed $t \in [b_n, 1-b_n]$, we have with probability at least $1 - 2 d^{-1}$ that
\begin{equation}
\label{eqn:large-dev-Sigma0}
|\hat\Sigma_{t,0} - \Sigma_{t,0}|_\infty \vee |\hat\Sigma_{t,1} - \Sigma_{t,1}|_\infty \le C \left( b_n^2 +\sqrt{\log{d} \over n b_n} \right),
\end{equation}
where the constant $C = C(\bar{\rho}, C_0, C_1)$ does not depend on $n, b_n$ and $d$.
\end{lem}

\begin{proof}[\textbf{Proof of Lemma \ref{lem:large-dev-Sigma0}}]
First, consider the marginal covariance matrix $m = 0$. For each $j,k=1,\cdots,p$, we write
\begin{equation*}
\hat\Sigma_{0,jk}(t) - \Sigma_{0,jk}(t) = \{\hat\Sigma_{0,jk}(t) - \E[\hat\Sigma_{0,jk}(t)]\} + \{\E[\hat\Sigma_{0,jk}(t)] - \Sigma_{0,jk}(t) \} = I + II.
\end{equation*}
For $|t_i - t| \le b_n$, by the assumption $\sup_{t \in [0,1]} |\ddot\Sigma_{0,jk}(t)| \le C_0$, Taylor's expansion yields that $\Sigma_{0,jk}(t_i) = \Sigma_{0,jk}(t) + \dot\Sigma_{0,jk}(t) (t_i-t) + O(b_n^2)$ for all $t \in [b_n,1-b_n]$. Note that for $\ell = 0,1,2$, we have
\begin{equation*}
\sup_{t \in [b_n, 1-b_n]} \left| \sum_{i=1}^n K\left( {t_i-t \over b_n} \right) (t_i-t)^\ell - n b_n \int_{-1}^1 K(v) v^\ell dv \right| = O(1).
\end{equation*}
Then, the bias part is controlled by
\begin{eqnarray}
\nonumber
II &=& \sum_{i=1}^n w(t,i) \Sigma_{0,jk}(t_i) - \Sigma_{0,jk}(t) \\
\label{eqn:large-dev-Sigma0-bias}
&=& \dot\Sigma_{0,jk}(t) \sum_{i=1}^n w(t,i) (t_i-t) + O(b_n^2) = O((n b_n)^{-1} + b_n^2).
\end{eqnarray}
Now, we deal with the stochastic part I. Let $Y_{i,jk} = X_{ij} X_{ik} - \E(X_{ij} X_{ik})$. Clearly, $\E(Y_{i,jk}) = 0$ and $I = \sum_{i=1}^n w(t,i) Y_{i,jk}$. Let $B_{i,0} = I_d$ and $B_{i,m} = A_i \cdots A_{i-m+1}$ for $m \ge 1$. Then, $\sup_{i} \rho(B_{i,m}) \le \bar{\rho}^m$ and we have the moving-average (MA) representation of $\vx_i = \sum_{m=0}^\infty B_{i,m} \ve_{i-m}$.  Let $\mbf\xi = (\ve_n^\top, \ve_{n-1}^\top, \cdots)^\top$, $W_t = \diag(w(t,n),\cdots,w(t,1))$ be $n \times n$ diagonal matrix, and
\begin{align*}
&\tilde{B}^{(j)}= \\
&\left(
\begin{array}{cccccccc}
B_{n,0,j\cdot} & B_{n,1,j\cdot}  & B_{n,2,j\cdot}  & \cdots & B_{n,n-1,j\cdot}  & B_{n,n,j\cdot}  & \cdots &   \\
0 & B_{n-1,0,j\cdot} & B_{n-1,1,j\cdot} & \cdots & B_{n-1,n-2,j\cdot} &  B_{n-1,n-1,j\cdot} & \cdots & \\
0 & 0 & B_{n-2,0,j\cdot} & \cdots & B_{n-2,n-3,j\cdot} & B_{n-2,n-2,j\cdot} & \cdots &  \\
\vdots & \vdots & \vdots & \ddots & \vdots & \vdots & \vdots \\
0 & 0 & 0 & \cdots & B_{1,0,j\cdot} &  B_{1,1,j\cdot} & \cdots & \\
\end{array} \right).
\end{align*}
Then, we can write $I = \mbf\xi^\top (\tilde{B}^{(j)})^\top W_t \tilde{B}^{(k)} \mbf\xi - \tr((\tilde{B}^{(j)})^\top W_t \tilde{B}^{(k)})$. By the Hanson-Wright inequality \cite[Theorem 1.1]{rudelsonvershynin2013a}, there exists a constant $C:=C(C_1)$ such that for all $x>0$
\begin{eqnarray*}
& (*) :=& \Prob \left(| \mbf\xi^\top (\tilde{B}^{(j)})^\top W_t \tilde{B}^{(k)} \mbf\xi - \tr((\tilde{B}^{(j)})^\top W_t \tilde{B}^{(k)}) | \ge x \right) \\
&\le& 2 \exp \left\{ -C \min\left[ {x^2 \over |(\tilde{B}^{(j)})^\top W_t \tilde{B}^{(k)}|_F^2 }, {x \over \rho((\tilde{B}^{(j)})^\top W_t \tilde{B}^{(k)}) } \right] \right\}.
\end{eqnarray*}
Observe that $(\tilde{B}^{(j)})^\top W_t \tilde{B}^{(k)}$ has the same nonzero eigenvalues as the $n \times n$ matrix $W_t^{1/2} \tilde{B}^{(k)} (\tilde{B}^{(j)})^\top W_t^{1/2}$. Therefore, by the Cauchy-Schwartz inequality, we get
\begin{eqnarray*}
|W_t^{1/2} \tilde{B}^{(k)} (\tilde{B}^{(j)})^\top W_t^{1/2}|_F^2 &=& \tr\left[W_t^{1/2} \tilde{B}^{(j)} (\tilde{B}^{(k)})^\top W_t \tilde{B}^{(k)} (\tilde{B}^{(j)})^\top W_t^{1/2} \right] \\
&=& \tr\left[ (\tilde{B}^{(k)})^\top W_t \tilde{B}^{(k)} (\tilde{B}^{(j)})^\top W_t \tilde{B}^{(j)} \right] \\
&\le& |(\tilde{B}^{(k)})^\top W_t \tilde{B}^{(k)}|_F \cdot |(\tilde{B}^{(j)})^\top W_t \tilde{B}^{(j)}|_F^2
\end{eqnarray*}
and by the definition of matrix spectral norm
\begin{eqnarray*}
\rho(W_t^{1/2} \tilde{B}^{(k)} (\tilde{B}^{(j)})^\top W_t^{1/2}|) &\le& \rho(W_t^{1/2} \tilde{B}^{(k)}) \rho((\tilde{B}^{(j)})^\top W_t^{1/2}) \\
&=& \rho^{1/2}((\tilde{B}^{(k)})^\top W_t \tilde{B}^{(k)}) \rho^{1/2}((\tilde{B}^{(j)})^\top W_t \tilde{B}^{(j)}).
\end{eqnarray*}
Therefore, (*) is bounded by
\begin{align*}
 2 \exp \left\{ {-C \min\left[ {x^2 \over \max_{j \le p} |W_t^{1/2} \tilde{B}^{(j)} (\tilde{B}^{(j)})^\top W_t^{1/2} |_F^2 }\right.}\right.,\\
 \left.\left.{x \over \max_{j \le p} \rho(W_t^{1/2} \tilde{B}^{(j)} (\tilde{B}^{(j)})^\top W_t^{1/2}) } \right] \right\}.
\end{align*}
For $i=1,\cdots,n$ and $l = 0,\cdots,n-i$, let
$$
\gamma^{(j)}_{i,l} = \sum_{m=0}^\infty [w(t,i) w(t,i+l)]^{1/2} |B_{i,m,j\cdot} B_{i+l,m+l,j\cdot}^\top|.
$$
By the Cauchy-Schwartz inequality and the spectral norm bound $\sup_{i} \rho(B_{i,m}) \le \bar{\rho}^m$, we have
\begin{eqnarray*}
\gamma^{(j)}_{i,l} &\le& [w(t,i) w(t,i+l)]^{1/2} \sum_{m=0}^\infty \left(\sum_{k=1}^p B_{i,m,jk}^2 \right)^{1/2} \left(\sum_{k=1}^p B_{i+l,m+l,jk}^2 \right)^{1/2} \\
&\le& [w(t,i) w(t,i+l)]^{1/2} \sum_{m=0}^\infty \bar{\rho}^m  \bar{\rho}^{m+l} = [w(t,i) w(t,i+l)]^{1/2} { \bar{\rho}^l \over 1- \bar{\rho}^2}.
\end{eqnarray*}
Then, it follows that
\begin{align*}
& \quad|W_t^{1/2} \tilde{B}^{(j)} (\tilde{B}^{(j)})^\top W_t^{1/2} |_F^2 \\
\le& \quad {\sum_{i=1}^n w(t,i)^2 \over (1-\bar{\rho}^2)^2} + 2 \sum_{l=1}^{n-1} \sum_{i=1}^{n-l} [w(t,i) w(t,i+l)]^{1/2} { \bar{\rho}^l \over (1- \bar{\rho}^2)^2} \\
\le& \quad  {2 \left[\sum_{l=0}^{n-1} \bar{\rho}^{2l} \right] \left[\sum_{i=1}^n w(t,i)^2 \right] \over (1- \bar{\rho}^2)^2 } \\
\le& \quad C(\bar{\rho}) \sum_{i=1}^n w(t,i)^2 = {C(\bar{\rho}) \over n b_n}
\end{align*}
and
\begin{align*}
& \quad\rho(W_t^{1/2} \tilde{B}^{(j)} (\tilde{B}^{(j)})^\top W_t^{1/2}) \\
\le& \quad 2 \left[ {1\over 1-\bar{\rho}^2} \max_{1\le i \le n} w(t,i) + \sum_{l=1}^{n-1} {\bar{\rho}^l \over 1-\bar{\rho}^2} \max_{1 \le i \le n-l} [w(t,i) w(t,i+l)]^{1/2} \right] \\
\le& \quad {2 \left[\sum_{l=0}^{n-1} \bar{\rho}^{l} \right] \left[\max_{1 \le i \le n} w(t,i) \right] \over 1- \bar{\rho}^2 } \le {C(\bar{\rho}) \over n b_n}.
\end{align*}

Therefore, we have that for any $x > 0$
\begin{equation}
\label{eqn:large-dev-Sigma0_stochastic}
\sup_{t \in [b_n 1-b_n]} \Prob(|\sum_{i=1}^n w(t,i) Y_{i,jk}| \ge x) \le 2 \exp[ - C n b_n \min(x^2, x) ].
\end{equation}
Now, by (\ref{eqn:large-dev-Sigma0-bias}) and using the union bound applied to (\ref{eqn:large-dev-Sigma0_stochastic}), we obtain that there exists a constant $C$ which only depends on $\bar{\rho}, C_0, C_1$ such that
$$
|\hat\Sigma_{t,0} - \Sigma_{t,0}|_\infty \le C \left( b_n^2 + \sqrt{\log{d} \over n b_n} \right)
$$
holds with probability $\ge 1 - 2 d^{-1}$. Similar argument applied to $m=1$ shows that the lag-one autocovariance matrix obeys the same bound in (\ref{eqn:large-dev-Sigma0}).
\end{proof}

\vspace{0.1in}
\begin{proof}[\textbf{Proof of Lemma \ref{lem:spatial_A}}] Denote $l = d/2$. For $0<\alpha<1$, note that
$$\underset{1 \le m\leq d}{\text{max}}\sum_{k=1}^d A_{mk}^\alpha=2\sum_{k=1}^{l-1} A_{lk}^\alpha + A_{ll}^\alpha=2\sum_{k=1}^{l-1} A_{lk}^\alpha+1.$$
Then, we have
\begin{align*}
\sum_{k=1}^{l-1}A_{lk}^\alpha=\quad&\sum_{k=1}^{l-1}\frac{1}{(1+{|l-k|^2}/{d^{2r}})^{\gamma\alpha}}\\
=\quad&\sum_{k=1}^{l-1}\frac{d^{2r\gamma\alpha}}{(d^{2r}+k^2)^{\gamma\alpha}}\\
\leq\quad&1 + \int_1^l \frac{d^{2r\gamma\alpha}}{(d^{2r}+x^2)^{\gamma\alpha}}dx\\
=\quad&1+d^r\int_{d^{-r}}^{d^{1-r}/2}\frac{1}{(1+x^2)^{\gamma\alpha}}dx\\
\leq\quad& C d^r \int_{1}^{d^{1-r}/2} x^{-2\gamma\alpha}dx,
\end{align*}
where $C > 0$ is a constant depending only on $r,\gamma,\alpha$. Then, it follows that
$$
\max_{1 \le m \le d} \sum_{k=1}^d A_{mk}^\alpha = \left\{
\begin{array}{cc}
O(d^r) & \text{if } 2 \gamma \alpha > 1 \\
O(d^r \log{d}) & \text{if } 2 \gamma \alpha = 1 \\
O(d^{r+(1-r)(1-2\gamma\alpha)}) & \text{if } 1 < 2 \gamma \alpha < 1 \\
\end{array}
\right. .
$$
So, $A\in {\cal G}_{\alpha}^{-}(s, M_d)$ for our choice of $s$ and $M_d$. Since $A$ is symmetric, $A \in {\cal G}_{\alpha}(s, M_d)$. Next, we show that $A \in {\cal G}_{\alpha,\beta}(s, M_d, L_d)$. Observe that $|\text{supp}(A)|=(3d-2)d/4$. By construction, $A_{mk}\leq u$ is equivalent to $\pi(h_m^\circ,h_k^\circ)\geq f^{-1}(u)=(u^{-\frac{1}{\gamma}}-1)^\frac{1}{2}$. Therefore, $A_{mk}\leq u$ if and only if $|m-k|\geq d^r(u^{-\frac{1}{\gamma}}-1)^\frac{1}{2}$. Since $d^r(u^{-\frac{1}{\gamma}}-1)^\frac{1}{2}$ is monotone decreasing with respect to $u$,
$$\exists \ubar{u}\in (0,u_0]\text{, s.t. }\frac{d}{2}-2<d^r(\ubar{u}^{-\frac{1}{\gamma}}-1)^\frac{1}{2}\leq \frac{d}{2}-1.$$
So we only need to consider $u\in [\ubar{u},u_0]$ because if $u\in (0,\ubar{u})$, then $|\{(m,k):0<A_{mk}\leq u \}|=|\{(m,k): |m-k|\geq d/2\}|=0$. Therefore, if we choose $L_d$ such that $L_d \ubar{u}^\beta \gtrsim 1$, then $A \in {\cal G}_{\alpha,\beta}(s, M_d, L_d)$. That is, $L_d = C \ubar{u}^{-\beta} = C d^{2(1-r)\gamma\beta}$.
\end{proof}

\vspace{0.1in}
\begin{proof}[\textbf{Proof of Theorem \ref{thm:rate-subGaussian-instantaneous}}] Put $\varepsilon = C(b_n^2 + \sqrt{(\log d) / (n b_n)})$. Recall that $A_i^\top = \Sigma_{i-1,0}^{-1} \Sigma_{i-1,1}$. By Lemma \ref{lem:large-dev-Sigma0}, we have, with probability at least $1 - 2 d^{-1}$, that $|\hat\Sigma_{t,0} - \Sigma_{t,0}|_\infty \le \varepsilon$ and $|\hat\Sigma_{t,1} - \Sigma_{t,1}|_\infty \le \varepsilon$ for any $t \in [b_n, 1-b_n]$. Then, with such high probability, the true transition matrix $A_i$ is feasible for the minimization program (\ref{eqn:tvvar-clime}) with $\tau \ge (1+M_d) \varepsilon$ because
\begin{eqnarray*}
|\hat\Sigma_{i-1,1} - \hat\Sigma_{i-1,0} A_i^\top|_\infty & = & |\hat\Sigma_{i-1,1} - \hat\Sigma_{i-1,0} \Sigma_{i-1,0}^{-1} \Sigma_{i-1,1}|_\infty \\
&\le& |\hat\Sigma_{i-1,1} - \Sigma_{i-1,1}|_\infty + |(I - \hat\Sigma_{i-1,0} \Sigma_{i-1,0}^{-1}) \Sigma_{i-1,1} |_\infty \\
&\le& \varepsilon + |\hat\Sigma_{i-1,0} - \Sigma_{i-1,0}|_\infty |A_i^\top|_{\ell^1} \le \tau.
\end{eqnarray*}
Hence, $|\hat{A}_i|_1 \le |A_i|_1$. But then
\begin{eqnarray*}
&&|\hat{A}_i^\top - A_i^\top|_\infty \\
&=& | \hat{A}_i^\top - \Sigma_{i-1,0}^{-1} \Sigma_{i-1,1} |_\infty \\
&=& | \Sigma_{i-1,0}^{-1} ( \Sigma_{i-1,0} \hat{A}_i^\top - \Sigma_{i-1,1}) |_\infty \\
&=& | \Sigma_{i-1,0}^{-1} ( \Sigma_{i-1,0} \hat{A}_i^\top - \hat{\Sigma}_{i-1,0} \hat{A}_i^\top + \hat{\Sigma}_{i-1,0} \hat{A}_i^\top - \hat{\Sigma}_{i-1,1} + \hat{\Sigma}_{i-1,1} - \Sigma_{i-1,1}) |_\infty \\
&\le& | \Sigma_{i-1,0}^{-1} ( \hat{\Sigma}_{i-1,0} - \hat{\Sigma}_{i-1,1}) \hat{A}_i^\top |_\infty + | \Sigma_{i-1,0}^{-1 }( \hat{\Sigma}_{i-1,0} \hat{A}_i^\top - \hat{\Sigma}_{i-1,1})|_\infty \\
&& \qquad + |\Sigma_{i-1,0}^{-1} (\hat{\Sigma}_{i-1,1} - \Sigma_{i-1,1}) |_\infty \\
&\le& | \Sigma_{i-1,0}^{-1}|_{\ell^1} (|\hat{\Sigma}_{i-1,0} - \hat{\Sigma}_{i-1,1}|_\infty |\hat{A}_i|_{\ell^1} + \tau + \varepsilon) \\
&\le& | \Sigma_{i-1,0}^{-1}|_{\ell^1} (\varepsilon |A_i|_{\ell^1} + \tau + \varepsilon) \\
&\le & 2 \tau |\Sigma_{i-1,0}^{-1}|_{\ell^1}.
\end{eqnarray*}
Let $u \ge 0$ and $T_j = \{m : |A_{i,jm}| \ge u\}, \quad j = 1,\cdots,d$. Define $D(u) = \max_{j \le d} \sum_{m=1}^d (|A_{i,jm}| \wedge u)$. Then, by the triangle inequality and the equivalence between (\ref{eqn:tvvar-clime}) and (\ref{eqn:tvvar-subproblems}), we have uniformly in $j = 1,\cdots,d$,
\begin{eqnarray*}
|[\hat{A}_i - A_i]_{j*}|_1 &\le& |[\hat{A}_i]_{j,T_j^c}|_1 + |[A_i]_{j,T_j^c}|_1 + |[\hat{A}_i - A_i]_{j,T_j}|_1 \\
&=& |[A_i]_{j*}|_1 - |[\hat{A}_i]_{j,T_j}|_1 + |[A_i]_{j,T_j^c}|_1 + |[\hat{A}_i - A_i]_{j,T_j}|_1 \\
&\le& 2 |[A_i]_{j,T_j^c}|_1 + 2 |[\hat{A}_i - A_i]_{j,T_j}|_1 \\
&\le& 2 |[A_i]_{j,T_j^c}|_1 + 4 \tau |\Sigma_{i-1,0}^{-1}|_{\ell^1} |T_j| \\
&\le& 2 D(u) (1 + 2 \tau |\Sigma_{i-1,0}^{-1}|_{\ell^1} / u).
\end{eqnarray*}
Choose $u = 2 \tau |\Sigma_{i,0}^{-1}|_{\ell^1}$. Since $D(u) \le s u^{1-\alpha}$ for $A_i \in {\cal G}_\alpha(s, M_d)$, it follows that
\begin{equation*}
|[\hat{A}_i - A_i]_{j*}|_1 \le 4 D(u) \le 4 s u^{1-\alpha} = C(\alpha) s (|\Sigma_{i,0}^{-1}|_{\ell^1} \tau)^{1-\alpha}.
\end{equation*}
The analysis of the other constraint is similar. Now, (\ref{eqn:rate-instantaneous-spectral}) and (\ref{eqn:rate-instantaneous-F}) are immediate  in view of $\rho^2(M) \le |M|_{\ell^1} |M|_{\ell^\infty}$ and $|M|_F^2 \le d |M|_\infty |M|_{\ell^1}$ for any $d \times d$ matrix $M$. 
\end{proof}
\vspace{0.1in}

\begin{proof}[\textbf{Proof of Theorem \ref{thm:partial-pattern-recovery}}]
Suppose $(m,k) \notin S_i$, i.e. $A_{i,mk}=0$. On the event $G = \{|\hat{A}_i-A_i|_\infty \le u_\sharp \}$, we have $|\hat{A}_{i,mk}| \le u_\sharp$ and $(m,k) \notin \hat{S}_i$. Therefore, by Theorem \ref{thm:rate-subGaussian-instantaneous}, we have $\Prob(\hat{S}_i \subset S_i) \ge 1-2d^{-1}$. On the other hand, suppose that $|A_{i,mk}|>2u_\sharp$, then on the event $G$, it follows from the triangle inequality that $|A_{i,mk}| - |\hat{A}_{i,mk}| \le |\hat{A}_{i,mk}-A_{i,mk}| \le u_\sharp$. Therefore, $|\hat{A}_{i,mk}| > u_\sharp$ and we have $\Prob(\{ (m, k) : |A_{i,mk}| > 2 u_\sharp \} \subset \hat{S}_{i}) \ge 1 - 2 d^{-1}$.
\end{proof}
\vspace{0.1in}

\begin{proof}[\textbf{Proof of Theorem \ref{eqn:pattern-recovery-consistency}}]
If $|S_i| = 0$, then the first claim is trivial. So we may assume that $S_i \neq \emptyset$. Suppose $(m,k) \notin S_i$, then on the event $G = \{|\hat{A}_i-A_i|_\infty \le u_\sharp \}$, we have $|\hat{A}_{i,mk}| \le u_\sharp$ and $(m,k) \notin \hat{S}_i$. Therefore, $\hat{S}_i \cap S_i^c = \emptyset$ on $G$ and $\Prob(\FPR_i=0) \ge 1-2d^{-1}$ by Theorem \ref{thm:rate-subGaussian-instantaneous}. On the other hand, let $N_i(u)=\{(m,k) : 0 < |A_{i,mk}| \le 2u\}$. For any $u>0$, $N_i(u) \subset S_i$ and $\hat{S}_i^c \cap S_i \subset N_i(u) \cup (\hat{S}_i^c \cap S_i \cap N_i(u)^c)$. Hence, $|\hat{S}_i^c \cap S_i| \le |N_i(u)| + |\hat{S}_i^c \cap S_i \cap N_i(u)^c|$. If $(m,k) \in  \hat{S}_i^c$ and $|A_{i,mk}|>2u_\sharp$, then $\hat{S}_i^c \cap S_i \cap N_i(u_\sharp)^c = \emptyset$ on $G$. Choose $u=u_\sharp$ and then on the event $G$, we have
$$
\FNR_i = {|\hat{S}_i^c \cap S_i| \over |S_i|} \le {|N_i(u_\sharp)| \over |S_i|} + {|\hat{S}_i^c \cap S_i \cap N_i(u_\sharp)^c| \over |S_i|} \le 2^\beta L_d u_\sharp^\beta,
$$
from which the theorem follows.
\end{proof}

\section*{Acknowledgements}
We thank Fang Han for providing the stock data in Section \ref{subsec:finance-data}. The high-performance computing work was done on the Illinois Campus Cluster Program (ICCP) at the University of Illinois at Urbana-Champaign.

\bibliographystyle{plain}
\bibliography{tvvar}

\appendix
\section{Detailed setup of the simulation studies} \label{App:setup_sim}
We use \texttt{sugm.generator()} to genereate sparce matrice from multivariate normal distributions with four graph structures-  hub, cluster, band and random. Besides taking the patterns of graphs, the \texttt{sugm.generator()} also takes other four critical input arguments which are \texttt{g}, \texttt{prob}, \texttt{v} and \texttt{u} shown in Table \ref{tab:arg}. The detailed setup of these fours parameters in our simulation is shown in Table \ref{tab:val_arg}.
\begin{table}[htbp]
  \centering
  \caption{The meanings of \texttt{sugm.generator()}'s arguments }
    \begin{tabular}{cl}
    \toprule
    Arguments & Meanings \\
    \midrule
     \texttt{g}   & Represents the number of hubs or clusters when the pattern is ``hub'' or \\
                  & ``cluster''\\
    \texttt{prob} & Represents the probability that an off-diagonal entry will be non-zero.\\
    \texttt{v}    & Assigns the values to the off-diagonal entries in the precision matrix and \\
                  & controls the magnitude of partial correlations with u. \\
    \texttt{u}    & A positive number which is added to the diagonal entries of the precision \\
                  & matrix.\\
    \bottomrule
    \end{tabular}%
  \label{tab:arg}%
\end{table}%
\begin{table}[htbp]
  \centering
  \caption{Detailed setup of \texttt{sugm.generator()}}
    \begin{tabular}{rrrrrrrrrr}
    \toprule
     & \multicolumn{4}{c}{Hub}       &  & \multicolumn{4}{c}{Cluster} \\
     \cmidrule(lr){2-5}
     \cmidrule(lr){7-10}
    \texttt{d}     & \texttt{g}     & \texttt{prob}  & \texttt{v}     & \texttt{u}     &  \texttt{d}     & \texttt{g}     & \texttt{prob}  & \texttt{v}     & \texttt{u} \\\hline
    20    & 8     & NULL  & 0.001 & 10    & 20    & 8     & NULL  & 0.001 & 10 \\
    30    & 10    & NULL  & 0.001 & 10    & 30    & 10    & NULL  & 0.001 & 10 \\
    40    & 15    & NULL  & 0.001 & 10    & 40    & 15    & NULL  & 0.001 & 10 \\
    50    & 20    & NULL  & 0.001 & 10    & 50    & 20    & NULL  & 0.001 & 10 \\\toprule
     & \multicolumn{4}{c}{Band}      &  & \multicolumn{4}{c}{Random} \\
     \cmidrule(lr){2-5}
     \cmidrule(lr){7-10}
     \texttt{d}     & \texttt{g}     & \texttt{prob}  & \texttt{v}     & \texttt{u}     & \texttt{d}     & \texttt{g}     & \texttt{prob}  & \texttt{v}     & \texttt{u}\\\hline
    20    & 1     & NULL  & 0.001 & 10    & 20    & NULL  & 0.001  & 0.001 & 10 \\
    30    & 1     & NULL  & 0.001 & 10    & 30    & NULL  & 0.001 & 0.001 & 10 \\
    40    & 1     & NULL  & 0.001 & 10    & 40    & NULL  & 0.001 & 0.001 & 10 \\
    50    & 1     & NULL  & 0.001 & 10    & 50    & NULL  & 0.001 & 0.001 & 10 \\
    \bottomrule
    \end{tabular}%
  \label{tab:val_arg}%
\end{table}%

\newpage
\section{Examples for four graph structures}\label{four_sparse_patterns}
\begin{figure}[H]
  \centering
  \caption{Examples of the baseline coefficient matrices in four graph structures. Zero entries are colored in grey. Positive entries and negative entries are colored in red and black respectively.}
  \subfloat[Hub]{\label{fig:hub_ex}\includegraphics[width=60mm]{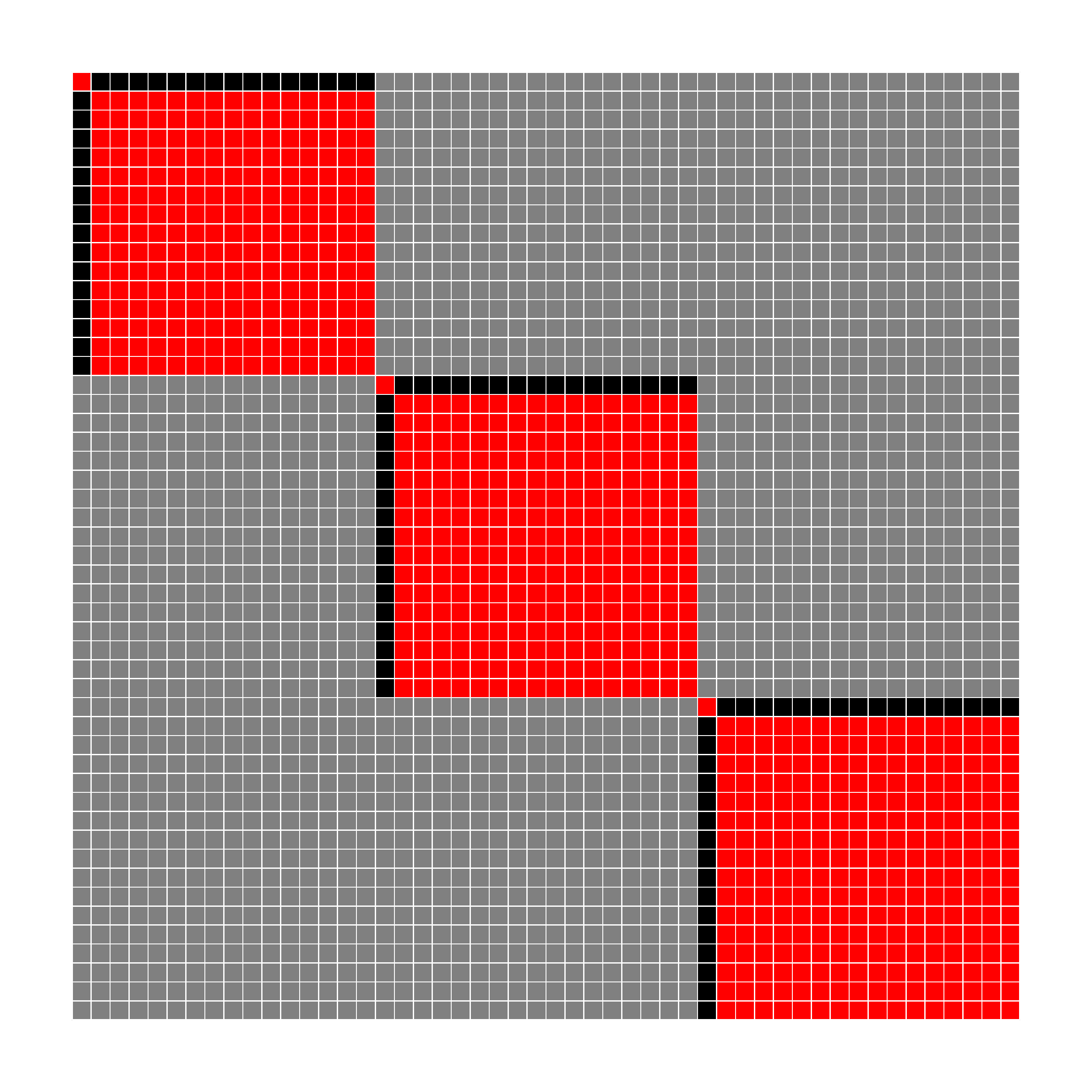}}
  \subfloat[Cluster]{\label{fig:cluster_ex}\includegraphics[width=60mm]{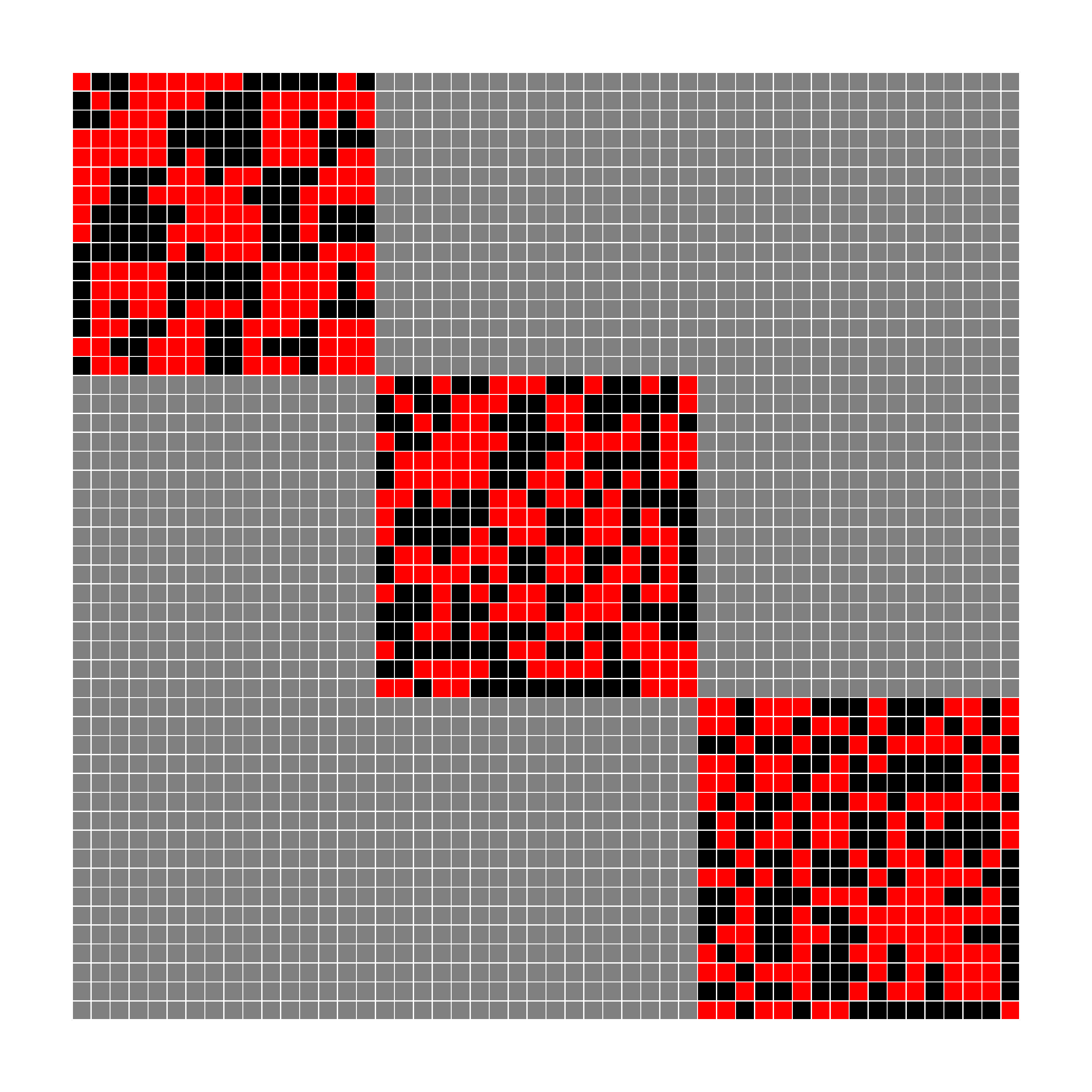}}
  \\
  \subfloat[Band]{\label{fig:band_ex}\includegraphics[width=60mm]{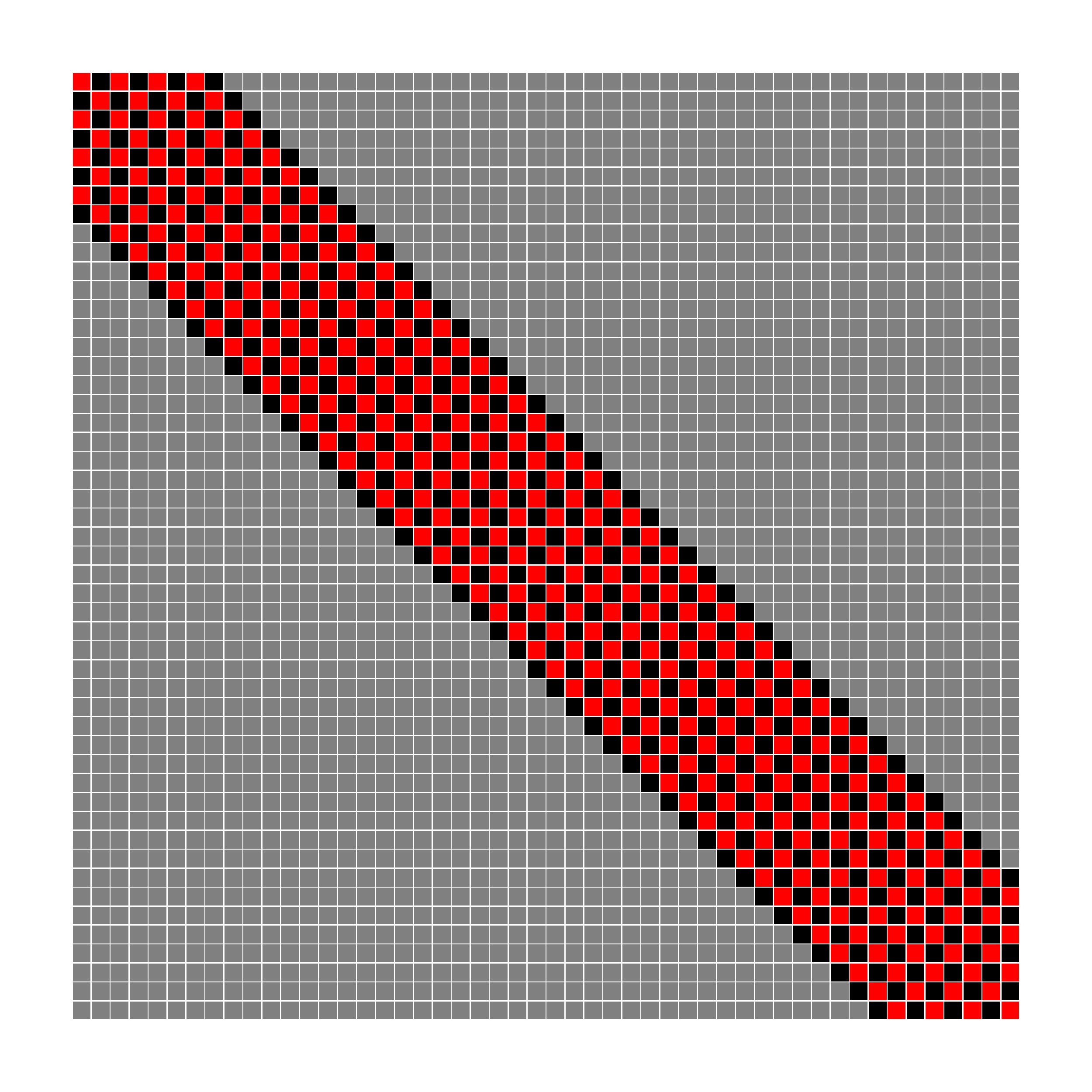}}
  \subfloat[Random]{\label{fig:random_ex}\includegraphics[width=60mm]{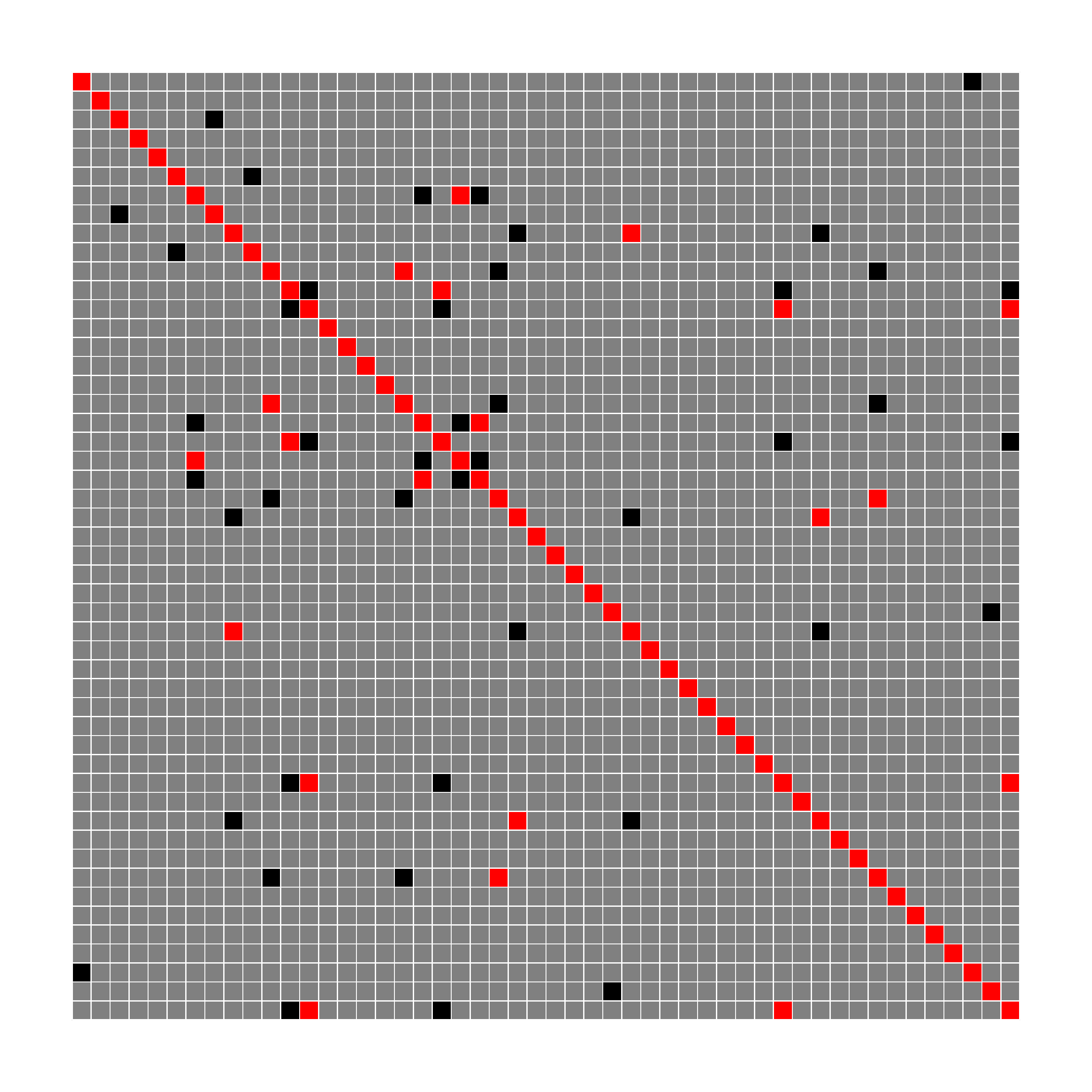}}
\end{figure}

\newpage
\section{ROC curves based on the true thresholding parameters $u_\sharp$}\label{App:true_ROC}
\begin{figure}[H]
  \centering
  \label{figur}\caption{ROC curves under different settings}
  \subfloat[Pattern is hub]{\label{fig:roc_hub_fix_tau}\includegraphics[width=60mm,height=60mm]{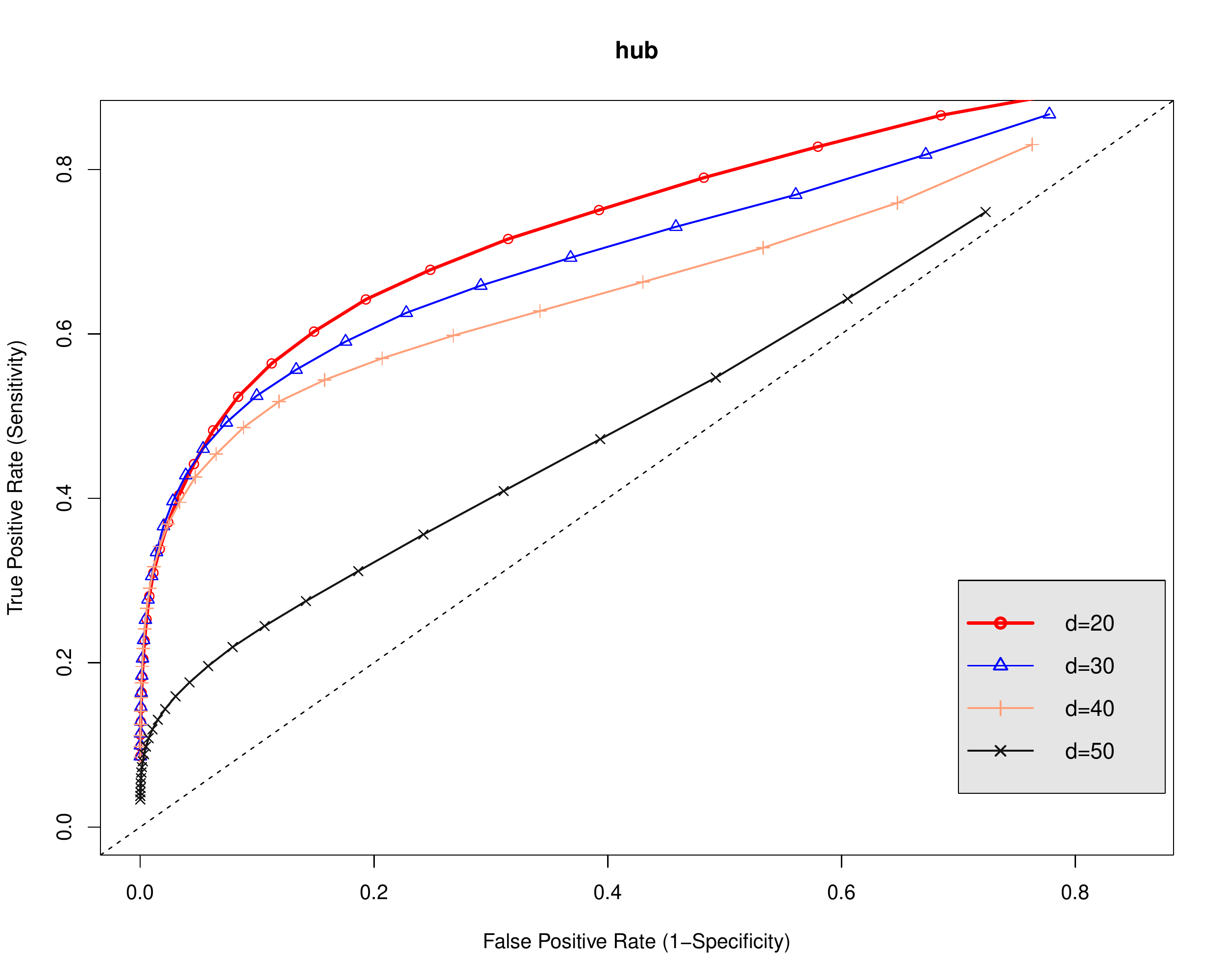}}
  \subfloat[Pattern is cluster]{\label{fig:roc_cluster_fix_tau}\includegraphics[width=60mm,height=60mm]{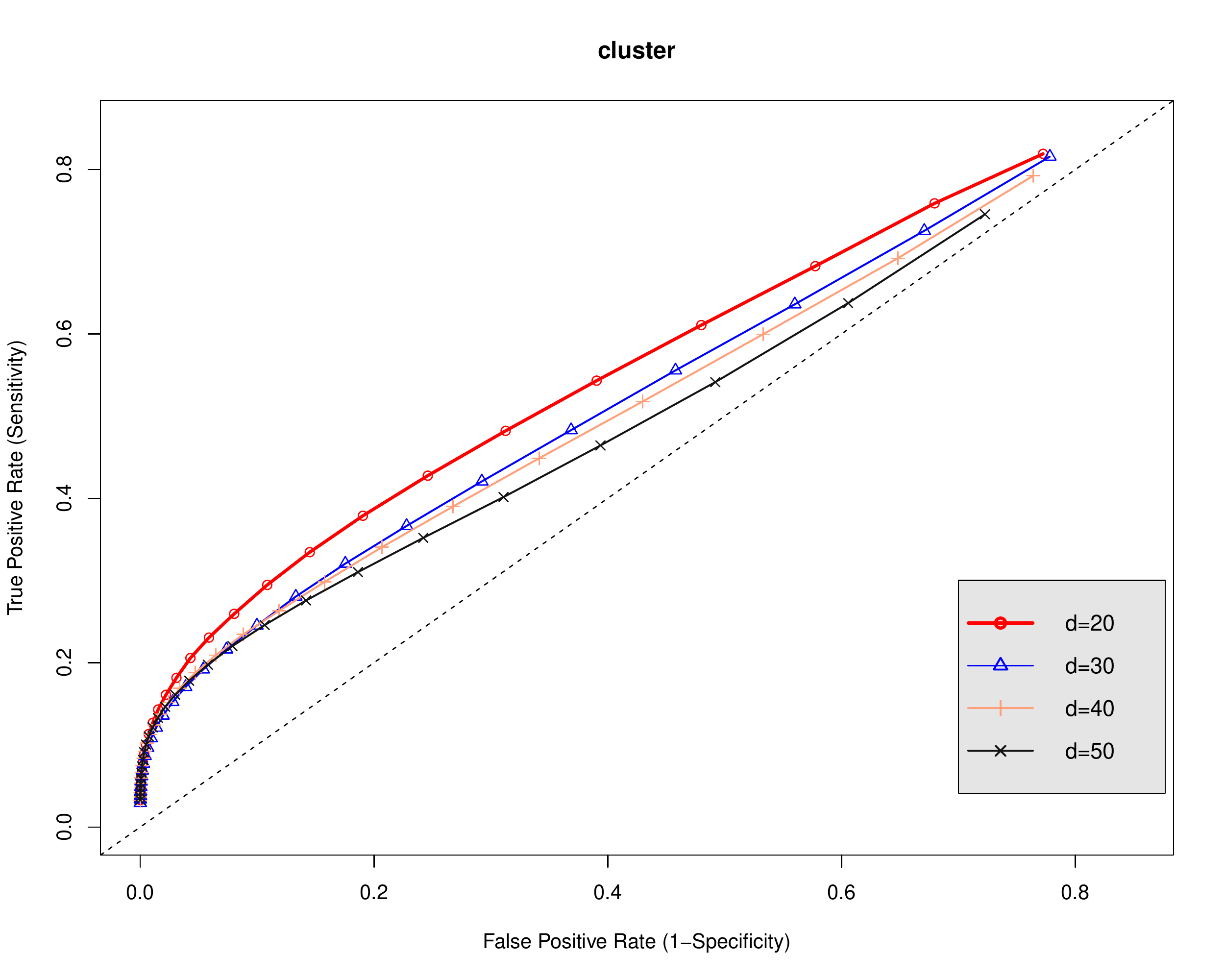}}
  \\
  \subfloat[Pattern is band]{\label{fig:roc_band_fix_tau}\includegraphics[width=60mm,height=60mm]{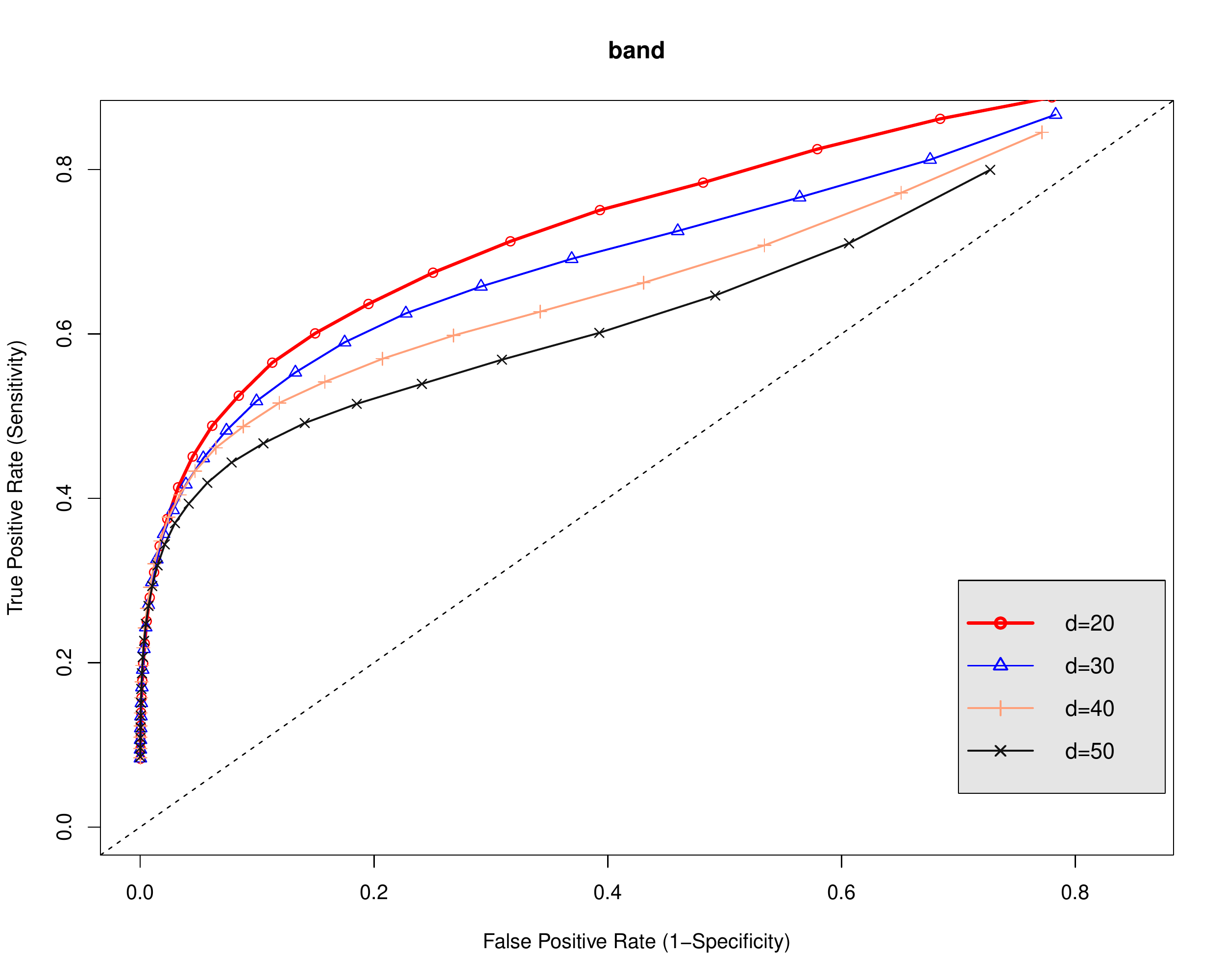}}
  \subfloat[Pattern is random]{\label{fig:roc_random_fix_tau}\includegraphics[width=60mm,height=60mm]{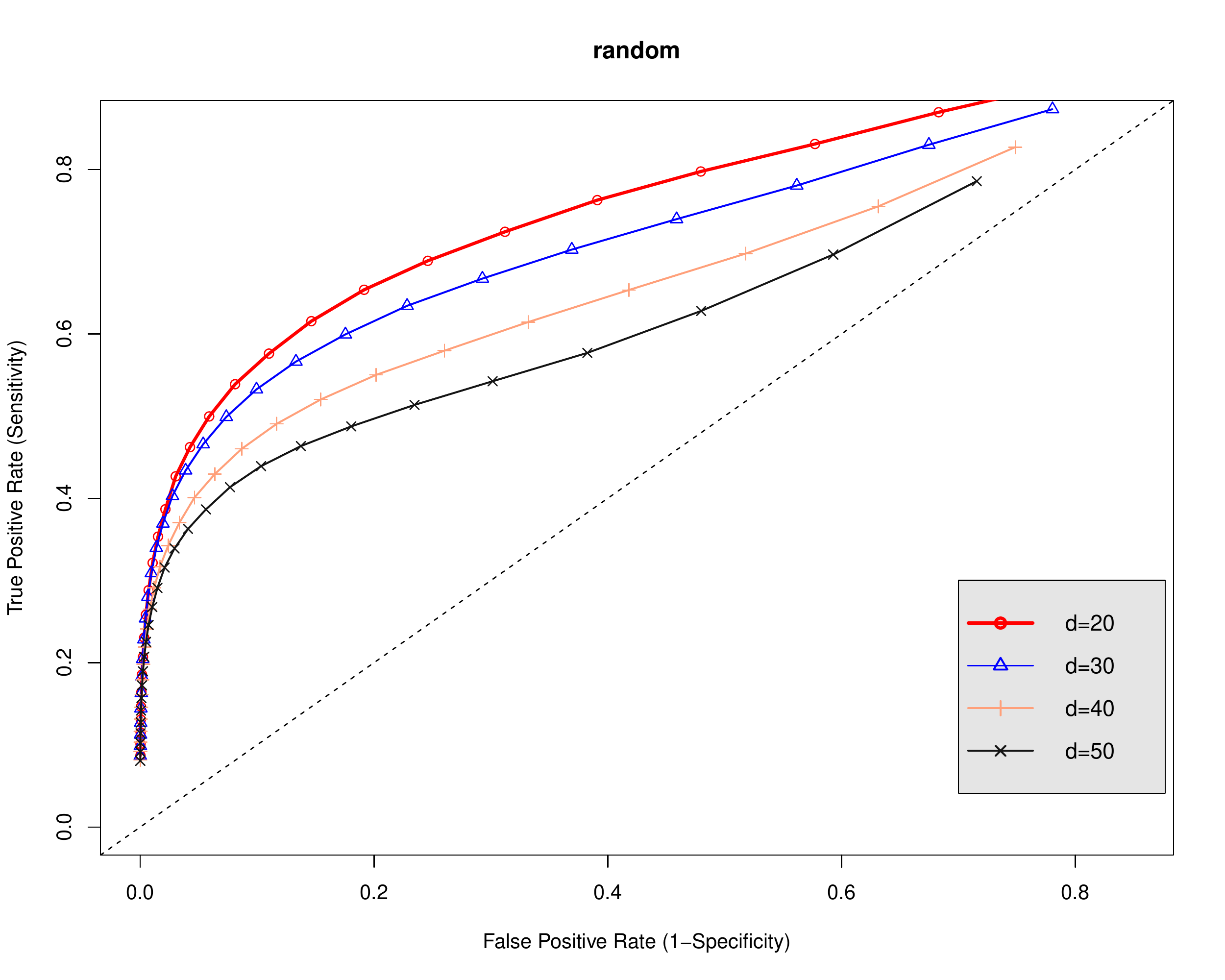}}
\end{figure}

\newpage
\section{Selected stocks and currencies in real data analysis}\label{App:sel_stocks_currencies}
\begin{table}[htbp]
  \centering
  \caption{Part of the 30 selected stocks}
    \begin{tabular}{lll}
    \toprule
          & Stock (Abb.) & Sector (Abb.)\\
    \midrule
    1     & Kellogg Co. (Kellogg) & Consumer Staples (CS) \\
    2     & Target Corp. (Target) & Consumer Discretionary (CD) \\
    3     & Boeing Company (Boeing) & Industrial (IND)\\
    4     & CME Group Inc. (CME) & Financials (FIN)\\
    5     & Prudential Financial (Prudential) & Financials (FIN)\\
    6     & Edison Int'l (Edison) & Utilities (UTI)\\
    7     & Lockheed Martin Corp. (Lockheed) & Industrial (IND)\\
    8     & PepsiCo Inc. (PepsiCo) & Consumer Staples (CS)\\
    9     & Hartford Financial Svc. GP (Hartford) & Financials (FIN)\\
    10    & Exxon Mobil Corp. (Exxon) & Energy (ENE)\\
    \bottomrule
    \end{tabular}%
  \label{tab:finance_selsto}%
\end{table}%

\begin{table}[htbp]
  \centering
  \caption{Selected currencies}
    \begin{tabular}{llll}
    \toprule
          & Currencies &       & Currencies \\
    \midrule
    1     & Euro / U.S. & 9     & Mexico / U.S. \\
    2     & U.K. / U.S. & 10    & South Korea / U.S. \\
    3     & Swiss / U.S. & 11    & India / U.S. \\
    4     & Australia / U.S. & 12    & Thailand / U.S. \\
    5     & New Zealand / U.S. & 13    & South Africa / U.S. \\
    6     & Canada / U.S. & 14    & Norway / U.S. \\
    7     & Singapore / U.S. & 15    & Sweden / U.S. \\
    8     & Brazil / U.S. &       &  \\
    \bottomrule
    \end{tabular}%
  \label{tab:exchange_selcur}%
\end{table}%

\end{document}